\documentclass[twoside]{irmaems}
\usepackage{amssymb} 
\usepackage{amsmath} 
\usepackage{latexsym}
\usepackage{amscd}
\usepackage{enumerate,rotating,multind,longtable,dsfont,pstricks}
\usepackage{wrapfig,psfig,graphicx,curves}

\renewcommand\theequation{\arabic{section}.\arabic{equation}}

\theoremstyle{definition} 

 \newtheorem{definition}{Definition}[section]
 \newtheorem{remark}[definition]{Remark}

\theoremstyle{plain}      

 \newtheorem{proposition}[definition]{Proposition}
 \newtheorem{theorem}[definition]{Theorem}
 \newtheorem{corollary}[definition]{Corollary}
 \newtheorem{lemma}[definition]{Lemma}

\newtheorem*{conjecture}{Conjecture}

\begin{document}

\title{Birkhoff's version of Hilbert's metric and its applications in analysis}

\author{Bas Lemmens\thanks{
Work partially supported by EPSRC grant EP/J008508/1} and Roger Nussbaum
\thanks{Work partially supported by NSF grant is NSF DMS 1201328} }

\address{
School of Mathematics, Statistics \& Actuarial Science\\ 
University of Kent,\\ Canterbury, Kent CT2 7NF, UK\\
email:\,\tt{B.Lemmens@kent.ac.uk}
\\[4pt]
Department of Mathematics, \\
Rutgers, The State University Of New Jersey, \\
110 Frelinghuysen Road, Piscataway, NJ 08854-8019, USA\\
email:\,\tt{nussbaum@math.rutgers.edu}
}

\maketitle

\begin{abstract} 
Birkhoff's version of Hilbert's metric is a distance between pairs of  rays in a closed cone, and  is closely related to Hilbert's classical cross-ratio metric. The version we discuss here was popularized by Bushell and can be traced back to the work of Garrett Birkhoff  and Hans Samelson. It has found numerous applications in mathematical analysis, especially in the analysis of linear, and nonlinear, mappings on cones. Some of these applications are discussed in this chapter. 

Birkhoff's version of Hilbert's metric provides a different perspective  on Hilbert geometries and naturally leads to infinite-dimensional generalizations.  
We illustrate this by showing some of its uses in the  geometric analysis of Hilbert geometries.
\end{abstract}

\begin{classification}
47H09, 54H20; 53C60
\end{classification}

\begin{keywords} 
Birkhoff's version of Hilbert's metric, Birkhoff's contraction coefficient,  Denjoy-Wolff type theorems, dynamics of non-expansive mappings, isometric embeddings, nonlinear mappings on cones. 
\end{keywords}

\tableofcontents   

\section{Introduction}\label{s-1}

In the nineteen-fifties Garrett Birkhoff  \cite{Bi} and Hans Samelson \cite{Sa} independently discovered that one can use Hilbert's metric and the contraction mapping principle  to give an elegant proof of the existence of a positive eigenvector for a variety of linear mappings that leave a cone in a real vector space invariant. Their results can be seen as a direct generalization of Perron's theorem \cite{Per1,Per2} concerning the existence and uniqueness of a positive eigenvector of square matrices with positive entries. 

In the past decades the ideas of Birkhoff and Samelson has been further developed by numerous authors. A partial list of contributors include, \cite{Bu1,Bu2,EN1, EN2,GG,KP,Kras1,KLS,Krau, KR,LNBook, Nmem1,Po,Tho,ZKP}. 
It has resulted in a remarkably detailed understanding of the eigenvalues, eigenvectors, and iterative behavior of a variety of linear, and nonlinear, mappings on cones. This body of work belongs to an area in mathematical analysis known  as nonlinear Perron-Frobenius theory, a recent introductory account of this field is given  in \cite{LNBook}.  

Recall that Hilbert's metric \cite{Hilbert} is defined as follows. 
Let $A$ be a real $n$-dimensional affine normed space, and denote the norm on the underlying vector space by $\|\cdot\|$. Consider a bounded, open, convex set 
$\Omega\subseteq A$. For $x,y\in \Omega$, let $\ell_{xy}$ denote the straight line through $x$ and $y$ in $A$, and denote the points of intersection of $\ell_{xy}$ and $\partial\Omega$ by $x'$ and $y'$, where $x$ is between $x'$ and $y$, and $y$ is between $x$ and $y'$, as in Figure \ref{fig:1}. 

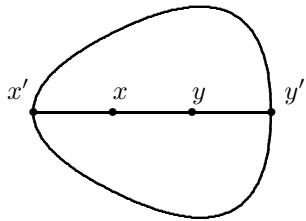
\begin{figure}[h]
\begin{center}
\begin{picture}(30, 70)  
\thicklines
  \closecurve(0,0, 60,30, 0,60)
  \put(-40,35){$x'$}
  \put(-30,30){\circle*{3}}
\put(0,35){$x$}
  \put(0,30){\circle*{3}}
\put(30,35){$y$}
  \put(30,30){\circle*{3}}
\put(65,35){$y'$}
  \put(60,30){\circle*{3}}
  \put(-30,30){\line(1,0){90}}
   \end{picture}
   \caption{Hilbert's cross-ratio metric}\label{fig:1}
  \end{center}  
\end{figure}
On $\Omega$, {\em Hilbert's metric}\index{Hilbert metric}  is defined by 
\begin{equation}\label{eq:1.1}
\delta(x,y) = \log \Big{(}\frac{\|x'-y\|}{\|x'-x\|} \frac{\|y'-x\|}{\|y'-y\|}\Big{)}
\end{equation}
for all $x\neq y$ in $\Omega$, and $\delta(x,x) =0$ for all $x\in \Omega$.
The metric space $(\Omega, \delta)$ is  called the  {\em Hilbert geometry}\index{Hilbert geometry}  on $\Omega$.

In mathematical analysis one uses an alternative version of Hilbert's metric, which is defined on a cone in a real vector space $V$. The version we will use here was popularized by Bushell  in \cite{Bu1}, and for simplicity we shall refer to  it as  {\em Birkhoff's version of Hilbert's metric}\index{Birkhoff's version of Hilbert's metric} . 
As we shall see,  Birkhoff's version ties together the partial ordering induced by $C$ and the metric. This idea has proved to be  very fruitful. 

The main objective of this survey is to discuss some of the applications of Birkhoff's version of Hilbert's metric in the analysis of nonlinear mappings on cones, and to  illustrate some of the advantages of Birkhoff's version by reproving several known geometric results for Hilbert geometries. Among other results we shall see  how Birkhoff's version can be used to give a simple proof of the well-known fact that the Hilbert geometry on an open $n$-simplex is isometric to and $n$-dimensional normed space, whose unit ball is a polytope with $n(n+1)$ facets. 

As many important applications of Hilbert's metric in analysis are in infinite-dimensional Banach spaces, we shall state some of the basic results in an infinite-dimensional setting and occasionally point out the difference between the infinite-dimensional and finite-dimensional cases. 

\section{Birkhoff's version of Hilbert's metric}
To define Birkhoff's version of Hilbert's metric we need to recall some elementary  concepts from the theory of partially ordered vector spaces. Let $V$ be a (possibly infinite-dimensional) real vector space. A subset $C$ of $V$ is called a {\em cone}\index{cone}  if 
\begin{enumerate}[(C1)] 
\item $C$ is convex, 
\item $\lambda C\subseteq C$ for all $\lambda \geq 0$, and
\item  $C\cap (-C) =\{0\}$.
\end{enumerate}
If $C\subseteq V$ satisfies (C1) and (C2) it is called a {\em wedge}. \index{wedge} A cone $C\subseteq V$ induces a partial ordering, $\leq_C$, on $V$ by 
\[
x\leq_C y \mbox{\quad if \quad }y-x\in C.
\]
If $C$ is merely a wedge, then $\leq_C$ is only a pre-order, as it may fail to be anti-symmetric. 

In analysis one often considers a closed cone $C$ in a Banach space $(V,\|\cdot\|)$. The cone $C$ is called {\em normal}\index{cone!normal}  if there exists a constant $\kappa>0$ such that $\|x\|\leq \kappa \|y\|$ whenever $0\leq_C x\leq_C y$. It is known, see \cite[Lemma 1.2.5]{LNBook}, that every closed cone in a finite-dimensional normed space is normal, but a closed cone in an infinite-dimensional Banach space may fail to be normal. Given a Banach space $(V,\|\cdot\|)$, we denote  the dual space of continuous linear functionals on $V$  by $V^*$.  If $C$ is a closed cone in a Banach space $(V,\|\cdot\|)$, the {\em dual wedge} \index{dual wedge} is given by, 
 \[
 C^*=\{\phi\in V^*\colon \phi(x)\geq 0\mbox{ for all }x\in V\}.
 \]
In general $C^*$ is only a wedge. However, if $C$ is a {\em total }\index{cone!total} cone, i.e., the closure of $C-C =\{x-y\colon x,y\in C\}$ is $V$, then $C^*$ is a cone. In particular, if $C$ is a closed cone in a finite-dimensional vector space $V$ and $C$ has nonempty interior, then $C^*$ is a closed cone in $V^*$ with nonempty interior. 
 
For $\phi \in C^*$ we let 
$\Sigma_\phi =\{x\in C\colon \phi(x) =1\}$.
It is not hard to show, see \cite[Lemma 1.2.4]{LNBook}, that if $C$ is a closed cone with nonempty interior in a finite-dimensional normed space, and $\phi$ in the interior of 
$C^*$, then $\Sigma_\phi$ is a nonempty, compact, convex subset of $V^*$.  However, in infinite dimensions it may happen that there does not exist $\phi\in C^*$ such that $\Sigma_\phi$ is bounded. A simple example is the Hilbert space, $\ell_2$, with the standard positive cone $C=\{(x_1,x_2,\ldots)\in\ell_2\colon x_i\geq 0\mbox{ for all }i\}$. 

As mentioned earlier, Birkhoff's version of Hilbert's metric ties together the partial ordering $\leq_C$ and the distance. To define it, let $C$  be a cone in a vector space $V$. For $x\in V$ and $y\in C$, we say that $y$ {\em dominates}\index{dominate}   $x$ if there exist $\alpha,\beta\in\mathbb{R}$ such that $\alpha y\leq_C x\leq_C \beta y$. In that case, we write  
\[
M(x/y) =\inf\{\beta\in\mathbb{R}\colon x\leq _C \beta y\}
\]
and
\[
m(x/y) = \sup\{\alpha\in\mathbb{R}\colon \alpha y\leq_C x\}.
\]
Obviously, if $x,y\in V$ are such that $y=0$ and $y$ dominates $x$,  then $x=0$, as $C$ is a cone. On the other hand, if $y\in C\setminus\{0\}$ and $y$ dominates $x$, then $M(x/y)\geq m(x/y)$. Using the domination relation one  obtains an equivalence relation on $C$ by $x\sim_C y$ if $y$ dominates $x$ and $x$ dominates $y$. The equivalence classes are called the {\em parts}\index{parts}   of $C$. Obviously $\{0\}$ is a part of $C$. Moreover, if $C$ is a closed cone with nonempty interior, $C^\circ$, in a Banach space, then $C^\circ$ is a part of $C$. The parts of a finite-dimensional cone are related to the faces of $C$. Indeed, if $C$ is a finite-dimensional closed cone, then it can be shown that the parts correspond to the relative interiors of the faces of $C$, see \cite[Lemma 1.2.2]{LNBook}. Recall that a {\em face}\index{face}  of a convex set $S\subseteq V$ is a subset  $F$ of $S$ with the property that if $x,y\in S$ and $\lambda x+ (1-\lambda)y\in F$ for some $0<\lambda <1$, then $x,y\in F$. The relative interior of a face $F$ is the interior of $F$ in its affine span.  

It is easy to verify that if $x,y\in C\setminus\{0\}$, then $x\sim_C y$ if, and only if, there exist $0<\alpha\leq \beta$ such that $\alpha y\leq_C x\leq_C\beta y$. Furthermore, 
\begin{equation}\label{eq:2.1}
m(x/y) = \sup\{\alpha>0\colon y\leq_C \alpha^{-1}x\}= M(y/x)^{-1}.
\end{equation}
Birkhoff's version of {\em Hilbert's metric}\index{Hilbert metric!Birkhoff's version}   on $C$ can now be defined as follows:
\begin{equation}\label{eq:2.2}
d(x,y) = \log\Big{(}\frac{M(x/y)}{m(x/y)}\Big{)}
\end{equation}
for all $x\sim_C y$ with $y\neq 0$, $d(0,0)=0$, and $d(x,y)=\infty$ otherwise. 
If $C$ is a closed cone in a Banach space, then $d$ is a genuine metric on the set of 
rays in each part of the cone as the following lemma shows. 
\begin{lemma}\label{lem:2.1.1} If $C$ is a cone in $V$, then for each $x\sim_C y\sim_C z$ with $y\neq 0$, 
\begin{enumerate}[(i)]
\item $d(x,y)\geq 0$,
\item $d(x,y)=d(y,x)$, 
\item $d(x,z)\leq d(x,y)+d(y,z)$,
\item $d(x,y) =d(\lambda x, \mu y)$ for all $\lambda,\mu>0$.
\end{enumerate}
Moreover, if $C$ is a closed cone in a Banach space $(V,\|\cdot\|)$, then $d(x,y) =0$ if, and only if, $x=\lambda y$ for some $\lambda\geq 0$. 
\end{lemma}
\begin{proof}
To prove the first assertion we note that for each $0<\alpha<m(x/y)$ and $0<M(x/y)<\beta$ we have  
\[
\alpha y\leq_C x\leq_C \beta y.
\]
It follows that $y\leq_C (\beta/\alpha)y$, and hence $\beta/\alpha \geq 1$. Thus $M(x/y)/m(x/y)\geq 1$ and hence $d$ is nonnegative. Furthermore, note that by (\ref{eq:2.1}), 
\begin{equation}
d(x,y) = \log\Big{(}M(x/y)M(y/x)\Big{)} = d(y,x),
\end{equation}
which shows that $d$ is symmetric. To show that $d$ satisfies the triangle inequality, we note that for each $0<\alpha <m(x/y)$ and $0<\gamma<m(y/z)$ we have  $\alpha y\leq_C x$ and $\gamma z\leq_C y$, and hence $\alpha\gamma z\leq_C x$. This implies that $m(x/z)\geq m(x/y)m(y/z)$. In the same way it can be shown that $M(x/z)\leq M(x/y)M(y/z)$. Thus, 
\[
\frac{M(x/z)}{m(x/z)} \leq \frac{M(x/y)M(y/z)}{m(x/y)m(y/z)}, 
\] 
which proves (iii). For $\lambda,\mu>0$, it  is easy to verify that 
\[
M(\lambda x/\mu y) = \frac{\lambda}{\mu}M(x/y)\mbox{\quad and\quad } 
m(\lambda x/\mu y) = \frac{\lambda}{\mu}m(x/y),
\]
which gives the fourth assertion. 

Finally, assume that $C$ is a closed cone in $(V,\|\cdot\|)$ and $x\sim_C y$ with $y\neq 0$. As $C$ is closed, $m(x/y)y\leq_C x\leq_C M(x/y)y$. If $d(x,y)=0$, then 
$M(x/y)m(x/y)^{-1}=1$, so we get that $y\leq_C m(x/y)^{-1}x \leq_C y$ from which we deduce that $x=m(x/y)y$. On the other hand, if $x=\lambda y$ for some $\lambda>0$, then $d(x,y)=0$ by the previous assertion.  
\end{proof}

To understand the relation with Hilbert's cross-ratio metric, $\delta$, given in (\ref{eq:1.1}),  we consider a cone $C$ in an 
$(n+1)$-dimensional real normed space $(V,\|\cdot\|)$ and we assume that the interior of $C$ is nonempty. Let $H\subseteq V$ be an $n$-dimensional affine hyperplane such that $\Omega_C:= H\cap C^\circ$ is a (relatively) open, bounded, convex set in $H$. 
\begin{theorem}\label{thm:2.2} 
The restriction of $d$ to $\Omega_C$ coincides with $\delta$. 
\end{theorem}
\begin{proof}
Consider $x\neq y$ in $\Omega_C$. Let $\alpha = m(x/y)= M(y/x)^{-1}$ and 
$\beta=M(x/y)$. 
As $C$ is closed, $\alpha y\leq_C x$ and $x\leq_C \beta y$. Write 
$u=x-\alpha y\in\partial C$ and $w=y-x/\beta \in\partial C$. 
Let $\ell_{xy}$ denote the straight line through $x$ and $y$ and denote the points of intersection with $\partial C$ by $x'$ and $y'$, as in Figure \ref{fig:2}.
\begin{figure}[h]
\begin{center}
\thicklines
\begin{picture}(150,100)(0,0)
\put(50,0){\line(-1,2){50}} 
\put(50,0){\line(1,1){100}}
\put(0,80){\line(1,0){150}} \put(30,80){\circle*{3.0}}
\put(30,85){$x$} \put(90,80){\circle*{3.0}}\put(90,85){$y$}
\put(50,0){\line(1,2){40}} 
\put(30,80){\line(-1,-2){10}}
\put(30,80){\line(1,-4){20}}

\put(90,80){\line(1,-4){8}}

\put(98,48){\circle*{3.0}}\put(103,42){$w$}
\put(20,60){\circle*{3.0}}\put(10,55){$u$}\put(10,80){\circle*{3.0}}\put(12,85){$x'$}
\put(130,80){\circle*{3.0}}\put(123,85){$y'$} 
\end{picture}
\caption{Theorem \ref{thm:2.2} \label{fig:2}}
\end{center}
\end{figure}
We know that there exist $\sigma,\tau>1$ such that 
\[
x'= y+\sigma(x-y)\mbox{\quad and\quad }y'=x+\tau(y-x).
\]

Let $\phi$ be a linear functional on $V$ such that $H =\{z\in V\colon \phi(z)=1\}$.  So, 
\[
y+\sigma(x-y) = x' =\frac{u}{\phi(u)}= \frac{x-\alpha y}{1-\alpha}, 
\]
which implies that $\alpha=(\sigma -1)/\sigma$. 
Similarly, 
\[
x+\tau(y-x) = y' =\frac{w}{\phi(w)}= \frac{y-x/\beta}{1-1/\beta}, 
\]
which gives $\beta =\tau/(1-\tau)$. 
Thus, 
\begin{equation}\label{eq:2.3} 
\frac{\|y-x'\|}{\|x-x'\|} = \frac{\sigma}{1-\sigma} =\frac{1}{\alpha} = M(y/x)
\end{equation}
and 
\begin{equation}\label{eq:2.4} 
\frac{\|x-y'\|}{\|y-y'\|} = \frac{\tau}{1-\tau} =\beta= M(x/y), 
\end{equation}
which shows that $d(x,y) =\delta(x,y)$ on $\Omega_C$. 

\end{proof}

\begin{remark}
From  equation (\ref{eq:2.3}) it follows that 
\[
\log \frac{\|x-y'\|}{\|y-y'\|} = \log M(x/y)
\]
for all $x\neq y$ in $\Omega_C$. The function $F_C\colon C^\circ\times C^\circ\to\mathbb{R}$ given by $F_C(x,y) =\log M(x/y)$ is called the {\em Funk (weak) metric}\index{Funk metric}   on $C^\circ$, see \cite{PT1,Wa2}. 
\end{remark} 

\begin{remark}
Birkhoff's version allows one to give a natural definition of Hilbert's metric on any open, convex, possibly unbounded, subset of an infinite-dimensional Banach space. Indeed, let $\Omega\subseteq (V,\|\cdot\|)$ be an open, convex set and assume that $\Omega$ contains no straight lines, i.e., for each $v\in\Omega$ there exists no $u\in V\setminus\{0\}$ such that $v+tu\in \Omega$ for all $t\in\mathbb{R}$. Let $W=\mathbb{R}\times V$ with norm $\|(s,v)\|_W= |s|+\|v\|$ for all $(s,v)\in W$.

Consider the following set in $W$: 
\[
C_\Omega =\{(s,sv)\in W\colon s>0\mbox{ and }v\in\Omega\}.
\]
It is easy to verify that $C_\Omega$ is open, and $\lambda C_\Omega\subseteq  C_\Omega$ for all $\lambda>0$.   Furthermore its closure, $\overline{C_\Omega}$, is a cone in $W$. Indeed, suppose that $w=(s,u)\in \overline{C_\Omega}$ and $-w\in \overline{C_\Omega}$. As $s\geq 0$, we know that $s=0$. Now let $w_k=(s_k,s_kv_k)\in C_\Omega$ be such that $w_k\to w$ and $v_k\in\Omega$ for all $k$.  Then for each $t>0$ fixed and $k$ sufficiently large, $0\leq s_kt\leq 1$. Thus, for each $v\in\Omega$ and $k$ large, we have 
\[
(1-s_kt)v +s_kt v_k\in\Omega.
\]
Note that $\|(1-s_kt)v +s_kt v_k -(v+tu)\|\leq s_kt\|v\| +t\|s_kv_k-u\|\to 0$  as $k\to\infty$,
so that $v+tu\in \overline{\Omega}$. But for $\epsilon >0$ small we have  $v+\epsilon u\in\Omega$, as $\Omega$ is open. As $t$ was arbitrary and $\Omega$ is convex, we deduce from \cite[Proposition 8.5]{Simon} that $v+tu\in\Omega$ for all $t\geq 0$. In the same way we can use the assumption that $-w\in \overline{C_\Omega}$ to show that $v+tu\in\Omega$ for all $t\leq 0$. This, however, contradicts the hypothesis  that $\Omega$ contains no lines. 

The restriction of $d$ to $\{(1,v)\in W\colon v\in\Omega\}$ is a genuine metric, which provides a natural definition of Hilbert's metric on $\Omega$.        
\end{remark}

\begin{remark}
From an analyst's point of view Hilbert's metric on cones has one disadvantage -- namely, it is only a metric on the pairs of rays in the cone rather than on the pairs of points in the  cone. In applications in analysis one therefore often also considers the following variant, which was introduced by Thompson \cite{Tho}. Given a closed cone $C$ in $(V,\|\cdot\|)$, {\em Thompson's metric}  is defined by 
\[
d_T(x,y) = \max\Big{\{}\log M(x/y), \log M(y/x)\Big{\}}
\]
for $x\sim_C y$ and $y\neq 0$, $d_T(0,0) =0$, and $d_T(x,y)=\infty$ otherwise. 
The reader can verify that $d_T$ is a genuine metric on each part of $C$. 

It is known, see \cite{Tho}, that if  $C$ is a closed normal cone in a Banach space 
$(V,\|\cdot\|)$ and $P\subset C$ is a part of $C$, then  
$(P,d_T)$ is a complete metric space whose topology coincides with the norm topology of the underlying space. Furthermore, if $q\colon C\setminus\{0\}\to (0,\infty)$ is a continuous homogeneous (degree $1$) function and $\Sigma_q=\{x\in C\colon q(x) =1\}$, then the metric space $(\Sigma_q\cap P,d)$ is also a complete metric space and its topology coincides with the norm topology. Particularly, interesting examples of homogeneous functions $q$ include $q \in C^*\setminus\{0\}$ and $q\colon x\mapsto \|x\|$. 
\end{remark} 

\subsection{Birkhoff's contraction ratio} 
The usefulness of Hilbert's metric lies in the fact that linear, but also certain nonlinear, mappings between cones are non-expansive with respect to this metric. 
Recall that a mapping $f$ from a metric space $(X,d_X)$ into a metric space $(Y,d_Y)$  is {\em non-expansive} if  
\[
d_Y(f(x),f(y))\leq d_X(x,y)\mbox{\quad for all }x,y\in X.
\]
It is said to be a {\em Lipschitz contraction} with constant $0\leq c<1$, if 
\[
d_Y(f(x),f(y))\leq cd_X(x,y)\mbox{\quad for all }x,y\in X.
\]

Note that if $C$ is a cone in $V$, $K$ is a cone in $W$, and $L\colon V\to W$ is a linear mapping, then $L(C)\subseteq K$ is equivalent to $Lx\leq_K Ly$ whenever $x\leq_C y$. A mapping $f\colon C\to K$ is called {\em order-preserving} if $x\leq_C y$ implies $f(x)\leq_K f(y)$. It is said to be {\em homogeneous of degree $r$} if $f(\lambda x) =\lambda^r f(x)$ for all $x\in C$ and $\lambda>0$. The following result is elementary, but  very useful. 

\begin{proposition}\label{prop:3.1} 
Let $C\subseteq V$ and $K\subseteq W$ be cones. If $f\colon C\to W$ is order-preserving  and homogeneous of degree $r>0$, then 
\begin{equation}\label{eq:3.1} 
M(f(x)/f(y))\leq M(x/y)^r\mbox{\quad and \quad} m(x/y)^r\leq m(f(x)/f(y))
\end{equation}
for all $x,y\in C$ with $x\sim_C y$.
\end{proposition}
\begin{proof}
As the proofs of the two inequalities are very similar, we shall only show the first one. 
If $x,y\in C$ and $x\sim_C y$, then for each $\beta>0$ with $x\leq_C\beta y$ we have that $f(x)\leq_K \beta^r f(y)$. So, $M(f(x)/f(y))\leq \beta^r$ for all $\beta>M(x/y)$, which gives $M(f(x)/f(y))\leq M(x/y)^r$. 
\end{proof}
As an immediate consequence we obtain the following  result. 
\begin{corollary}\label{cor:3.2} 
Let $C\subseteq V$ and $K\subseteq W$ be cones. If $f\colon C\to K$ is order-preserving and homogeneous (of degree $1$), then $f$ is non-expansive with respect to $d$ on each part of $C$.
\end{corollary}
\begin{remark}
It is interesting to note that every order-preserving homogeneous mapping  $f$ is also 
non-expansive with respect to  Thompson's metric. In fact, it can be shown 
(see \cite[Lemma 2.1.7]{LNBook}) that an order-preserving mapping $f\colon C\to K$  is non-expansive under Thompson's metric if, and only if, $f$ is {\em subhomogeneous}\index{subhomogeneous map}, i.e., $\lambda f(x)\leq_K f(\lambda x)$ for all $x\in C$ and $0\leq \lambda\leq 1$. 
We should also remark that it is also easy to show that  every order-reversing,   homogeneous degree $-1$ mapping is non-expansive under $d$, see \cite[Corollary 2.1.5]{LNBook}.
 \end{remark}
From Corollary \ref{cor:3.2} we  see that every linear mapping $L\colon V\to W$ with $L(C)\subseteq K$ is non-expansive with respect to $d$ on each part of $C$. Furthermore, if $L$ is invertible and $L(C)=K$, then $L$ must be an isometry. In fact, in that case we have 
\[
M(Lx/Ly)= M(x/y)\mbox{\quad and\quad } m(Lx/Ly)=m(x/y)
\]
for all $x\sim _C y$ in $C$.  
The non-expansiveness of linear mappings, $L\colon C\to C$, on cones provides a way to analyze the eigenvalue problem, $Lx=\lambda x$, on $C$ using contraction mapping arguments. 

Given cones $C\subseteq V$, $K\subseteq W$ and a linear mapping $L\colon V\to W$ with $L(C)\subseteq K$, the {\em Birkhoff contraction ratio} of $L$ is defined by 
\[
\kappa(L) =\inf\{c\geq 0\colon d(Lx,Ly)\leq cd(x,y)\mbox{ for all }x\sim_C y\mbox{ in }C\}.
\]
Moreover, the {\em projective diameter} \index{projective diameter}  of $L$ is given by 
\[
\Delta(L)=\sup\{d(Lx,Ly)\colon x,y\in C\mbox{ with }Lx\sim_K Ly\}.
\]
\begin{theorem}[Birkhoff]\label{thm:3.4} \index{Birkhoff's theorem}  Let $C$ be a cone in a vector space $V$ and $K$ be a cone in a vector space $W$. If $L\colon V\to W$ is a linear mapping with $L(C)\subseteq K$, then 
\[
\kappa(L) = \tanh\Big{(}\frac{1}{4}\Delta(L)\Big{)},
\]
where $\tanh(\infty) =1$.
\end{theorem}
So, if $\Delta(L)<\infty$, then $L$ is a Lipschitz contraction on each part of $C$, with contraction constant $\tanh(\Delta(L)/4)<1$. In numerous cases one can prove that $\Delta(L)<\infty$. Indeed, Birkhoff showed this for matrices with positive entries and  for certain integral operators with positive kernels. IN case the linear map $L$ is given by a positive $m\times n$ matrices $A=(a_{ij})$, so $a_{ij}>0$ for all $i$ and $j$, there exists the following well-known  
explicit formula, see \cite[Appendix A]{LNBook}, 
\[
\Delta(A) = \max_{i,j}d(Ae_i, Ae_j) =\log\Big{(}\max_{i,j,p,q} \frac{a_{pi}a_{qj}}{a_{pj}a_{qi}}\Big{)}<\infty,
\]
where $e_1,\ldots, e_n$ denote the standard basis vectors in $\mathbb{R}^n$. 

\begin{remark}
In \cite{Ho1,Ho2} Hopf proved a result  closely  related to Birkhoff's Theorem \ref{thm:3.4}. Hopf was apparently unaware of Birkhoff's work and did not use Hilbert's metric. The results of Birkhoff and Hopf have been extended and their connections have been unraveled by numerous authors including, \cite{Bau, Bu1,Bu2, EN1,EN2}. A detailed exposition of these works can be found in \cite[Appendix A]{LNBook}.     
\end{remark} 

Birkhoff's Theorem \ref{thm:3.4} can be used to prove the existence and uniqueness 
of positive eigenvectors of continuous linear mappings that leave a closed, normal  cone in  a Banach space invariant. In fact, one has the following slightly more general set up. 
A mapping $L$ from a cone $C\subseteq V$ into a cone $K\subseteq W$ is said to be 
{\em cone linear}\index{ cone linear},   if $L(\alpha x+\beta y) = \alpha Lx+\beta Ly$ for all $\alpha,\beta\geq 0$ and $x,y\in C$. There exist examples, see \cite{Bon},  of cone linear mappings $L$ from a closed cone $C$ in a Banach space  $V$ into a closed cone $K$ in a Banach space $W$ such that $L$ is continuous on $C$, but its linear extension to $\mathrm{cl}(C-C)$ is not continuous, even if $\mathrm{cl}(C-C) =V$. For cone linear mappings there exists the following result, see \cite[Theorem A.7.1]{LNBook}.
\begin{theorem}\label{thm:3.5}
Let $C$ be a closed normal cone in a Banach space $V$, and let $L\colon C\to C$ be a cone linear mapping, which is continuous at $0$. If there exists an integer $p\geq 1$ such that $\Delta(L^p)<\infty$ and $L^{p+1}(C)\neq \{0\}$, then $L$ has a unique eigenvector $v\in C$, with $\|v\|=1$, such that $Lv= r_C(L) v$, where 
\[
r_C(f) =\lim_{k\to\infty}\|L^k\|_C^{1/k}
\] 
is the cone spectral radius of $L$ and \[
\|L^k\|_C =\sup\{\|L^kx\|\colon x\in C\mbox{ with }\|x\| =1\}.\] 
Moreover, if we let $c=\tanh(\Delta(L^p)/4)<1$, then 
\[
d(L^{kp}x,v)\leq c^kd(x,v)\mbox{\quad  for all }x\in C\setminus\{0\}.
\]
\end{theorem}

\subsection{An application to Perron-Frobenius operators}

In this section we briefly discuss an application of Birkhoff's  contraction ratio, or, more precisely, Theorem \ref{thm:3.5}, to  so-called Perron-Frobenius operators. These operators play a central role in the study of Hausdorff dimensions of invariant sets given by graph directed iterated function systems, see \cite{Nu3, NPVL}.  For simplicity we restrict attention here to operators arising from iterated function systems rather than the more general case of graph directed iterated function systems.  However, the same ideas can be applied in the graph directed case.

Let $(S,\rho)$ be a bounded, complete metric space with positive diameter. Let $C_b(S)$ denote the Banach space of bounded, continuous, real-valued functions, $f\colon S\to \mathbb{R}$, with $\|f\|=\sup_{s\in S} |f(s)|$. Write $\Delta =\mathrm{diam} (S)$, so $0<\Delta<\infty$. For $0<\lambda\leq 1$ and $M>0$ define
\[
K(M,\lambda) =\{f \in C_b(S)\colon f(s)\leq f(t) e^{M\rho(s,t)^\lambda}\mbox{ for all }s,t\in S\}.
\]
It can be shown, see \cite[Lemma 3.2]{NPVL}, that $K(M,\lambda)$ is a closed normal cone in $C_b(S)$. By using the assumption that $\Delta>0$, it is easy to verify that if $f\in K(M,\lambda)$, then $f(s)\geq 0$ for all $s\in S$, and $f(s) =0$ for all $s\in S$ whenever $f(t)=0$ for some $t\in S$. Furthermore for $f\in K(M,\lambda)$ we have 
\begin{equation}\label{eq:2.2.1}
\sup_{s\in S} f(s)\leq e^{M\Delta^\lambda}\Big{(} \inf_{s\in S} f(s)\Big{)}.
\end{equation}

Now fix $M_0>0$ and $0<\lambda\leq 1$ and let $b_i\in K(M_0,\lambda)$ for $i\in\mathbb{N}$. Assume that $\theta_i\colon S\to S$ are Lipschitz mappings for $i\in\mathbb{N}$, with 
\begin{equation}\label{eq:2.2.2} 
\mathrm{Lip}(\theta_i) = \sup\Big{\{}\frac{\rho(\theta_i(s),\theta_i(t))}{\rho(s,t)}\colon s\neq t\mbox{ in } S\Big{\}} \leq c<1. 
\end{equation}
Assume, in addition, that there exists $s^*\in S$ such that 
\begin{equation}\label{eq:2.2.3} 
\sum_i b_i(s^*)<\infty.
\end{equation}
Now we define a {\em Perron-Frobenius operator}\index{Perron-Frobenius operator} ,   $L\colon C_b(S)\to C_b(S)$, by 
\begin{equation}\label{eq:2.2.4} 
(Lf)(t) = \sum_i b_i(t) f(\theta_i(t))\mbox{\quad for }t\in S.
\end{equation}
Obviously, 
\[
\|Lf\|= \sup_{t\in S} |\sum_i b_i(t)f(\theta_i(t))|\leq \sum_i b_i(s^*) e^{M_0\Delta^\lambda}\|f\|<\infty\]
for all $f\in C_b(S)$, and hence $L$ is a continuous linear mapping from $C_b(S)$ to $C_b(S)$. We shall now see how we can use Birkhoff's Theorem \ref{thm:3.4} to prove that $L$ has a unique eigenvector, $v\in K(M_2,\lambda)$, with $\|v\|=1$, for all $M_2> M_0/(1-c^\lambda)$. Here $c$ is given in (\ref{eq:2.2.2}). Furthermore, for each $g\in K(M_2,\lambda)\setminus\{0\}$, 
\begin{equation}\label{eq:2.2.5} 
\lim_{k\to\infty} \Big{\|}\frac{L^kg}{\|L^k g\|} - v\Big{\|} = 0. 
\end{equation}

We will need the following lemma. 
\begin{lemma} \label{lem:2.2.1}
Let $d_2$ denote the Hilbert metric on $K(M_2,\lambda)$. If $0<M_1<M_2$, then $K(M_1,\lambda)\setminus\{0\}$ is a part of $K(M_2,\lambda)$ and 
\[
\sup\{d_2(f,g)\colon f,g\in K(M_1,\lambda)\setminus\{0\}\}<\infty. 
\] 
\end{lemma}
 \begin{proof}
 Let $\leq_2$ denote the partial ordering induced by $K(M_2,\lambda)$ on $C_b(S)$. 
 We need to show that there exists $0<\alpha\leq \beta$ such that 
 \begin{equation}\label{eq:2.2.6} 
 \alpha f\leq_2 g\leq_2 \beta f\mbox{\quad for all }f,g\in K(M_1,\lambda).
 \end{equation} 
 As $d_2(\sigma f,\tau g)=d_2(f,g)$ for all $\sigma,\tau>0$, we may as well assume that  
 $\|f\|=\|g\|=1$. Furthermore, by interchanging the roles of $f$ and $g$ in (\ref{eq:2.2.6}), we see that it suffices to prove that there exists $\alpha>0$ such that $\alpha f\leq_2 g$ for all $f,g\in K(M_1,\lambda)$ with $\|f\|=\|g\|=1$.  Recall that $\alpha f\leq_2 g$ is equivalent to
 \begin{equation}\label{eq:2.2.7} 
 g(s) -\alpha f(s) \leq (g(t)-\alpha f(t))e^{M_2\rho(s,t)^\lambda}\mbox{\quad for all } s,t\in S.
 \end{equation}
 Since $f,g\in K(M_1,\lambda)$, we know that 
 \begin{equation}\label{eq:2.2.8} 
 g(s)\leq g(t)e^{M_1\rho(s,t)^\lambda}\mbox{\quad and \quad}f(s)\leq f(t)e^{M_1\rho(s,t)^\lambda}
 \end{equation}
 for all $s,t\in S$. This implies that 
 \[
 e^{-M_1\Delta^\lambda}\leq g(t)\leq 1\mbox{\quad and \quad} e^{-M_1\Delta^\lambda}\leq f(t)\leq 1,
 \]
for all $t\in S$,  from which we deduce that 
\begin{equation}\label{eq:2.2.9} 
e^{-M_1\Delta^\lambda}\leq \frac{g(t)}{f(t)}\leq e^{M_1\Delta^\lambda}\mbox{\quad for all } t\in S.
\end{equation}

For convenience we shall assume that 
\begin{equation}\label{eq:2.2.10} 
\alpha <e^{-M_1\Delta^\lambda},
\end{equation}
which ensures that $g(t) -\alpha f(t)>0$ for all $t\in S$. This allows us  to rewrite (\ref{eq:2.2.7}) as
\begin{equation}\label{eq:2.2.11} 
\frac{g(s)-\alpha f(s)}{g(t) -\alpha f(t)}\leq e^{M_2\rho(s,t)^\lambda}\mbox{\quad for all }s,t\in S. 
\end{equation}

Now using (\ref{eq:2.2.8}) we see that for $s,t\in S$,
\begin{eqnarray*} 
\frac{g(s)-\alpha f(s)}{g(t) -\alpha f(t)} & \leq & 
 \frac{g(t) e^{M_1\rho(s,t)^\lambda} - \alpha f(t)e^{-M_1\rho(s,t)^\lambda}}{g(t)-\alpha f(t)}\\
 & \leq &  \frac{(g(t)/f(t)) e^{M_1\rho(s,t)^\lambda} - \alpha e^{-M_1\rho(s,t)^\lambda}}
 {(g(t)/f(t))-\alpha}.
\end{eqnarray*}
This implies for $s,t\in S$ that 
\begin{equation}\label{eq:2.2.12} 
 \frac{g(s)-\alpha f(s)}{g(t) -\alpha f(t)} \leq e^{M_1\rho(s,t)^\lambda} + \frac{\alpha(e^{M_1\rho(s,t)^\lambda} - e^{-M_1\rho(s,t)^\lambda})}{e^{-M_1\rho(s,t)^\lambda} -\alpha} 
\end{equation}

Write $r= M_2\rho(s,t)^\lambda$, so $r\geq 0$ and $M_1\rho(s,t)^\lambda = \mu r$, where $\mu=M_1/M_2<1$. From (\ref{eq:2.2.12}) we see that (\ref{eq:2.2.11}) is satisfied if $0<\alpha< e^{-M_1\Delta^\lambda}$, can be chosen so that 
\begin{equation}\label{eq:2.2.13}
\Big{(}\frac{\alpha}{e^{-M_1\Delta^\lambda}-\alpha}\Big{)}\Big{(}\frac{e^{\mu r} -e^{-\mu r}}{e^r - e^{\mu r}}\Big{)}\leq 1.
\end{equation}

Consider the function $\theta\colon [0,\infty)\to\mathbb{R}$ given by
\[
\theta (r) = \frac{e^{\mu r}- e^{-\mu r}}{e^r - e^{\mu r}}\mbox{ for }r>0\mbox{ \quad and\quad  }             \theta(0) = \frac{2\mu}{1-\mu}. 
\]
 Observe that $\theta$ is continuous on $[0,\infty)$ and $\lim_{r\to\infty} 
 \theta(r)=0$. 
 A simple calculation gives 
 \[
 \theta'(r) = \frac{-(1-\mu)e^{(\mu +1)r} -2\mu +(\mu +1)e^{(1-\mu)r}}{(e^r - e^{\mu r})^2}.
 \]
 By a power series expansion, for $r>0$ we have 
 \begin{multline*}
 -(1-\mu)e^{(\mu +1)r} -2\mu +(\mu+1)e^{(1-\mu)r}   = \\
  (1-\mu^2) \Big{(} -\sum_{j=2}^\infty \frac{(\mu+1)^{j-1}r^j}{j!} + 
  \sum_{j=2}^\infty \frac{(1-\mu)^{j-1}r^j}{j!}\Big{)}<0.
 \end{multline*}
 
 Thus, $\theta$ achieves its maximum on $[0,\infty)$ at $0$.  Let  
$\kappa =(M_2-M_1)/(M_2+M_1)$
and assume that $\alpha$ satisfies $0<\alpha\leq \kappa e^{-M_1\Delta^\lambda}$. By 
the previous remarks equation (\ref{eq:2.2.13}) will be satisfied if 
\[
\Big{(}\frac{\kappa}{1-\kappa}\Big{)}\Big{(}\frac{2\mu}{1-\mu}\Big{)}\leq 1.
\]
However, a calculation shows that the left-hand-side of the above expression actually equals $1$, so that we can take $\alpha=\kappa e^{-M_1\Delta^\lambda}$, and the proof is complete. 
 \end{proof}
 Note that we have actually shown that for each $f,g\in K(M_1,\lambda)$ with $\|f\|=\|g\|=1$, we can choose 
 \[
 \alpha =\Big{(}\frac{M_2-M_1}{M_2+M_1}\Big{)} e^{-M_1\Delta^\lambda}
 \mbox{\quad 
 and \quad}
 \beta =\Big{(}\frac{M_2+M_1}{M_2-M_1}\Big{)} e^{M_1\Delta^\lambda}.
 \]
 This implies that 
 \[
 d_2(f,g) \leq 2 \log \frac{M_2+M_1}{M_2-M_1} + 2 e^{M_1\Delta^\lambda}.
 \]
 for all $f,g\in K(M_1,\lambda)$.
 
 We shall also need the following result. 
 \begin{lemma}\label{lem:2.2.2} 
 Let $L\colon C_b(S)\to C_b(S)$ be given by (\ref{eq:2.2.4}), and assume that at least one $b_i$ is not identically $0$. If $M_2>M_0/(1-c^\lambda)$ and $M_1 = M_0 +c^\lambda M_2$, then 
 \[ 
 L(K(M_2,\lambda)\setminus\{0\})\subseteq K(M_1,\lambda)\setminus\{0\}.
 \]
  \end{lemma}
  \begin{proof}
  If $f\in K(M_2,\lambda)\setminus\{0\}$, then $f\circ\theta_i\in K(c^\lambda M_2,\lambda)$, because 
  \[
  f(\theta_i(s))\leq f(\theta_i(t))e^{M_2\rho(\theta_i(s),\theta_i(t))^\lambda}
 		\leq f(\theta_i(t))e^{c^\lambda M_2\rho(s,t)^\lambda}
  \]
  for all $s,t\in S$. As $b_i\in K(M_0,\lambda)$, we have 
  \[
  b_i(s)\leq b_i(t)e^{M_0\rho(s,t)^\lambda}\mbox{\quad for all }s,t\in S.
  \]
  Combining these two inequalities gives
  \[
  b_i(s)f(\theta_i(s))\leq b_i(t)f(\theta_i(t))e^{(M_0+c^\lambda M_2)\rho(s,t)^\lambda}
  \]
 for all $s,t\in S$, which shows that $b_i(\cdot)f(\theta_i(\cdot))\in K(M_1,\lambda)$ for all $i\in\mathbb{N}$. Since $K(M_1,\lambda)$ is a closed cone, we find that $L(f)\in K(M_1,\lambda)$. 
 
 It remains to show that $L(f)\neq 0$ for $f\in K(M_2,\lambda)\setminus\{0\}$. 
 Note that as $f\neq 0$ and $f\in K(M_2,\lambda)$, we know that $f(s)>0$ for all $s\in S$. As each $b_i\in K(M_0,\lambda)$,  $b_i(s)\geq 0$ for all $s\in S$. Furthermore, there exists $j$ such that $b_j\neq 0$, and hence $b_i(t)>0$ for all 
 $t \in S$. Thus, $L(f)(t)\geq b_j(t)f(\theta_j(t))>0$ for all $t\in S$, which completes the proof.
  \end{proof}
 We can now use Birkhoff's Theorem \ref{thm:3.4} to prove the following result. 
 \begin{theorem}\label{thm:2.2.2} Let $L\colon C_b(S)\to C_b(S)$ be given by (\ref{eq:2.2.4}), and assume that at least one $b_i$ is not identically $0$. If $M_2>M_0/(1-c^\lambda)$,  then there exists a unique $v\in K(M_2,\lambda)$ with $\|v\|=1$, such that 
$ Lv = \|Lv\| v$.
 Furthermore, for each $g\in K(M_2,\lambda)\setminus\{0\}$ we have 
 \begin{equation}\label{eq:limit}
 \lim_{k\to\infty} \Big{\|}\frac{L^kg}{\|L^kg\|} - v\Big{\|}=0.
 \end{equation}
 \end{theorem}
 \begin{proof}
 Let $\leq_2$ denote the partial ordering from $K(M_2,\lambda)$ and denote the restriction of $L$ to $K(M_2,\lambda)$ by $L_0$. Note that it follows from Lemmas \ref{lem:2.2.1} and \ref{lem:2.2.2} that the projective diameter $\Delta(L_0)<\infty$.  Now let $g\in K(M_2,\lambda)\setminus\{0\}$. It follows from Lemma \ref{lem:2.2.2} that  $L_0g\in K(M_1,\lambda)$ and $L_0g \neq 0$. As $K(M_1,\lambda)\subseteq K(M_2,\lambda)$ the same applies to $L_0g$.  So, $L_0^2g \in K(M_1,\lambda)\setminus\{0\}$, and hence $L_0^2(K(M_2,\lambda))\neq \{0\}$. It now follows from 
 Theorem \ref{thm:3.5} that there exists a unique $v\in K(M_2,\lambda)$, with $\|v\|=1$, such that $Lv= \|Lv\| v$, and for each $g\in K(M_2,\lambda)\setminus\{0\}$,
 Equation (\ref{eq:limit}) holds.
 \end{proof}
 \begin{remark}
 As Theorem \ref{thm:2.2.2} holds for all $M_2> M_0/(1-c^\lambda)$, the unique eigenvector $v$ in Theorem \ref{thm:2.2.2} belongs to $K(M_0/(1-c^\lambda),\lambda)$.  It can be shown that the elements of $K(M,\lambda)$ are H\"older continuous functions with H\"older exponent $\lambda$, and hence $v$ is H\"older with H\"older exponent $\lambda$. 
 \end{remark}

\section{Special cones} 
Particularly important examples of  finite-dimensional cones arising in analysis are the {\em standard positive cone},\index{standard positive cone},
\[
\mathbb{R}^n_+ =\{x\in\mathbb{R}^n\colon x_i\geq 0\mbox{ for }i=1,\ldots,n\}
\]
and the {\em Lorentz cone}, \index{Lorentz cone},
\[
\Lambda_{n+1}=\{(s,x)\in\mathbb{R}\times\mathbb{R}^n\colon s^2\geq x_1^2+\cdots + x_n^2\mbox{ and }s\geq 0\}.
\]
Other interesting examples are cones of linear operators, such as the cone $\Pi_n(\mathbb{R})$ consisting of positive semi-definite $n\times n$ matrices in the real vector space $\mathrm{Sym}_n$  of all $n\times n$ symmetric matrices. 
The cones mentioned above are all  examples of so-called symmetric cones. Recall that a cone $C$,  with nonempty interior, in a finite-dimensional inner product space $(V,\langle\cdot \mid\cdot\rangle)$ is called {\em self-dual} if $C^*=C$. Furthermore $C$ is called {\em homogeneous} if $\mathrm{Aut}(C)=\{A\in \mathrm{GL}(V)\colon A(C)=C\}$ acts transitively on $C^\circ$.  The interior of a self-dual homogeneous cone is called a {\em symmetric cone}. 

It is well-known that the symmetric cones in finite dimensions are in one-to-one correspondence with the interiors of the cones of squares of Euclidean Jordan algebras. This fundamental result was proved  by Koecher \cite{Koe} and Vinberg \cite{Vin}. A detailed exposition of the theory of symmetric cones can be found in \cite{FK}.  Recall that a {\em Euclidean Jordan algebra} \index{Euclidean Jordan algebra} is a finite-dimensional real inner-product space $(V,\langle\cdot \mid\cdot\rangle)$ equipped with a bilinear product $(x,y)\mapsto xy$ from $V\times V$ into $V$ such that for each $x,y\in V$
\begin{enumerate}
\item[(J1)] $xy=yx$, 
\item[(J2)] $x(x^2y)=x^2(xy)$, 
\item[(J3)] for each $x\in V$, the linear mapping $L(x)\colon V\to V$ defined by $L(x)y =xy$ 
satisfies 
\[
\langle L(x)y\mid z\rangle = \langle y\mid L(x)z\rangle\mbox{\quad for all }y,z\in V. 
\]
\end{enumerate}
Note that a Euclidean Jordan algebra is commutative, but in general not associative. 
For example,  $\mathrm{Sym}_n$ can be equipped with the usual inner product $\langle A\mid B\rangle =\mathrm{tr}(BA)$ and (Jordan) product $A\circ B = (AB+BA)/2$. In this case the interior of cone of squares in $\mathrm{Sym}_n$ is precisely $\Pi_n(\mathbb{R}^\circ)$, as 
\[
\Pi_n(\mathbb{R})^\circ = \{A\in\mathrm{Sym}_n\colon A\mbox{ is positive definite}\}= 
\{A^2\colon A\in\mathrm{Sym}_n\}^\circ.
\] 

When studying a Hilbert geometry $(\Omega,\delta)$, it is often useful to view $\Omega$ as the (relative) interior of the intersection of a cone in a vector space and a hyperplane $H$, and use Birkhoff's version of Hilbert's metric. In many interesting cases it gives an alternative formula to compute the distance and provides  additional tools to analyze  the Hilbert geometry. To illustrate this we need to recall some basic concepts. 

Let $C$ be a closed cone with nonempty interior in a finite-dimensional vector space $V$. For $u\in C^\circ$ define $\Sigma^*_u=\{\phi \in C^*\colon \phi(u)=1\}$, which is a compact convex set in $V^*$. Let $\mathcal{E}_u$ denote the closure of the set of extreme points of $\Sigma_u^*$. Recall that $\phi\in\Sigma_u^*$ is an extreme point if there exist no $\psi,\rho \in\Sigma_u^*$  such that $\phi =\lambda \psi+(1-\lambda)\rho$ for some $0<\lambda <1$. A basic result in convex analysis, due to Minkowski, says that each compact convex set in a finite-dimensional vector space is the closed convex hull of its extreme points, see \cite[Corollary 13.5]{Rock}. 
This result, and its infinite-dimensional extension, is usually called the Krein-Milman theorem.  

It is easy to show that if $C$ is a closed cone with nonempty interior and 
$u \in C^\circ$, then for $x,y\in V$ we have $x\leq_C y$ if, and only if, $\phi(x)\leq \phi(y)$ for all $\phi \in \mathcal{E}_u$. Indeed, it follows from the Krein-Milman theorem that $\phi(x)\leq\phi(y)$ for all $\phi\in C^*$ is equivalent to $\phi(x)\leq \phi(y)$ for all $\phi\in \mathcal{E}_u$. 
Now if $x\nleq_C y$, then $y-x\not\in C$. So, by the Hahn-Banach separation theorem, there exist $\alpha\in\mathbb{R}$ and $\psi\in V^*$ such that $\psi(y-x)<\alpha$ and $\psi(z)>\alpha$ for all $z\in C$. As $\lambda z\in C$ for all $\lambda\geq 0$ and $z\in C$, we see that $\psi(z)\geq 0$ for all $z\in C$, and hence $\psi\in C^*$. Obviously, $\psi(0)=0$, and hence $\alpha<0$. Thus, $\psi(y)<\psi(x)$.  On the other hand, $x\leq_C y$ implies $\phi(x)\leq \phi(y)$ for all $\phi\in C^*$. So, we have the following result. 
\begin{lemma}\label{lem:4.1} 
If $C$ is a closed cone with nonempty interior in a finite-dimensional vector space, and $u\in C^\circ$, then for $x\in V$ and $y\in C^\circ$ we have  
\begin{equation}\label{eq:4.1}
M(x/y) =\max_{\phi\in\mathcal{E}_u}\frac{\phi(x)}{\phi(y)}
\mbox{\quad and\quad }
m(x/y) =\min_{\phi\in\mathcal{E}_u}\frac{\phi(x)}{\phi(y)}.
\end{equation}
\end{lemma}

\begin{remark}
The identities in (\ref{eq:4.1}) also hold for closed  cones in infinite-dimensional topological vector spaces. Indeed, let $C$ be a closed cone, with nonempty interior in a topological vector space $V$. Given $u\in C^\circ$, one can define the {\em order unit norm}, $\|\cdot\|_u$, by 
\[
\|x\|_u =\inf\{\alpha>0\colon -\alpha u\leq_C x\leq_C\alpha u\}
\]
for all $x\in V$. With respect to this norm each linear functional, $\phi \colon V \to\mathbb{R}$, with $\phi ( V )\subseteq [0,\infty)$, is continuous, because $|\phi(x)|\leq \phi(u)$ for all $x\in V$ with $\|x\|_u\leq 1$. Now let $C^*=\{\phi\colon V\to\mathbb{R} \mid \phi (V)\subseteq [0,\infty)\}$ be the dual of $C$ in 
$(V,\|\cdot\|_u)^*$ and let $\Sigma_u^*=\{\phi\in C^*\colon \phi(u)=1\}$. Note that 
$\Sigma_u^*$ is a closed subset of the unit ball, $B^*$, of $(V,\|\cdot\|_u)^*$. 
So, $\Sigma_u^*$ is a weak-$*$ compact set, as $B^*$ is weak-$*$ compact by the Banach-Alaoglu theorem. If we now let $\mathcal{E}_u$ be the weak-$*$ closure of the set of extreme points of $\Sigma_u^*$, then $\mathcal{E}_u$ is weak-$*$ compact and the equalities in (\ref{eq:4.1}) hold. 

\end{remark}
\subsection{Simplicial cones} 

A cone $C$ in an $n$-dimensional vector space $V$ is called a {\em simplicial cone} \index{simplicial cone} if there exist $v_1,\ldots,v_n\in V$, linearly independent, such that 
\[
C = \{\sum_{i=1}^n \lambda_i v_i\colon \lambda_i\geq 0\mbox{ for all }i\}.
\] 
A basic example is the standard positive cone $\mathbb{R}^n_+$. We can apply Lemma \ref{lem:4.1} to derive an explicit formula for Hilbert's metric on $(\mathbb{R}^n_+)^
\circ$, and use this formula to give a simple proof of the well known fact that the Hilbert geometry on the open {\em standard $(n-1)$-dimensional simplex}, 
$\Delta_{n-1}^\circ =\{x\in (\mathbb{R}^n_+)^\circ\colon \sum_{i=1}^n x_i =1\}
$, 
is isometric to an $(n-1)$-dimensional normed space, see \cite{dlH,Nmem1,Pha}. 

Indeed, Lemma \ref{lem:4.1}, or a direct simple argument,  gives
\[
M(x/y) =\max_i x_i/y_i\mbox{\quad and\quad }m(x/y) = \min_j x_j/y_j
\]
for all $x.y\in(\mathbb{R}^n_+)^\circ$. 
Now consider the mapping $\mathrm{Log}\colon (\mathbb{R}^n_+)^\circ\to\mathbb{R}^n$ given by, $\mathrm{Log}(x) = (\log x_1,\ldots,\log x_n)$ for $x\in(\mathbb{R}^n_+)^\circ$, and equip $\mathbb{R}^n$ with the {\em variation norm},  \index{variation norm},
\[
\|w\|_{\mathrm{var}}= \max_i w_i -\min_j w_j.
\]
Note that 
$
\log M(x/y) =\log \max_i x_i/y_i =\max_i \log x_i -\log  y_i$. 
Likewise, we see that $\log m(x/y) = \min_j \log x_j -\log y_j$.
Thus, the mapping, $x\mapsto \mathrm{Log}(x)$, is an isometry from $((\mathbb{R}^n_+)^\circ, d)$ onto $(\mathbb{R}^n,\|\cdot\|_{\mathrm{var}})$.

If we let $H=\{x\in (\mathbb{R}^n_+)^\circ\colon x_n =1\}$, then $(H,d)$ is a genuine metric space, which is isometric to $(\Delta_{n-1}^\circ,d)$ by Lemma \ref{lem:2.1.1}(iv). Clearly $\mathrm{Log}(H)=\{x\in\mathbb{R}^n\colon x_n=0\}$, which we can identify with $\mathbb{R}^{n-1}$ by projecting out the last coordinate. It follows that $(\Delta_{n-1}^\circ,d)$ is isometric to the $(n-1)$-dimensional normed space $(\mathbb{R}^{n-1},\|\cdot\|_H)$, where 
\[
\|x\|_H =\max\{x_1,\ldots,x_{n-1},0\}-\min\{x_1,\ldots,x_{n-1},0\}.
\]
If we now use Theorem \ref{thm:2.2}, we arrive at the following result. 
\begin{theorem}\label{thm:simplex}
The Hilbert geometry $(\Delta_n^\circ,\delta)$ is isometric to $(\mathbb{R}^n,\|\cdot\|_H)$.
\end{theorem}
The unit ball of $\|\cdot\|_H$ in $\mathbb{R}^n$ is a polytope with $n(n+1)$ facets. In fact, it is a hexagon, when $n=2$, which was already known to Phadke \cite{Pha}, and  a rhombic-dodecahedron, when $n=3$. 

If $\Delta= \mathrm{conv}\{v_1,\ldots,v_{n+1}\}$ is a simplex with nonempty interior in an $n$-dimensional vector space $V$, then $C=\{(\lambda v,\lambda)\in V\times \mathbb{R}\colon v\in \Delta\mbox{ and }\lambda\geq 0\}$ is a simplicial cone  in $V'=V\times\mathbb{R}$.  Note that $v'_1 =(v_1,1),\ldots,v_{n+1}'=(v_{n+1},1)$ is a basis for $V'$ and  $C=\{\sum_i\lambda_iv'_i\colon \lambda_i\geq 0\mbox{ for all }i\}$. Moreover, the linear mapping $A\colon V'\to\mathbb{R}^{n+1}$ given by 
\[
A(\sum_i\mu_iv'_i) =\sum_i \mu_ie_i,
\]  
where $\mu_1,\ldots,\mu_{n+1}\in\mathbb{R}$ and $e_1,\ldots,e_{n+1}$ are the standard basis vectors in $\mathbb{R}^{n+1}$, is invertible and maps $C$ onto $\mathbb{R}^{n+1}_+$. So, $A$ is an isometry from $(C^\circ,d)$ onto $((\mathbb{R}^{n+1}_+)^\circ,d)$. 
By combining this observation with Theorem \ref{thm:simplex} it is easy to show that the Hilbert geometry on $\Delta^\circ$ is isometric to $(\mathbb{R}^n,\|\cdot\|_H)$. In \cite{FoK} Foertsch and Karlsson showed that the only Hilbert geometry isometric to a normed space is the  one on an open simplex. Thus, the following result holds. 
\begin{theorem}
An $n$-dimensional Hilbert geometry is isometric to a normed space if and only if its domain is an open $n$-dimensional simplex. In that case it is isometric to $(\mathbb{R}^n,\|\cdot\|_H)$. 
\end{theorem}

\subsection{Polyhedral cones} 
 In this subsection we shall see how we can use Birkhoff's version of Hilbert's metric on polyhedral cones to show that the Hilbert geometry on the interior of a polytope with $m$ facets can be isometrically embedded into the normed space $(\mathbb{R}^{m(m-1)/2},\|\cdot\|_\infty)$, where $\|z\|_\infty =\max_i |z_i|$ is the {\em supremum norm}. It turns out the polytopal Hilbert geometries are the only ones that can be isometrically embedded into a finite-dimensional normed space, see \cite{Ber,CV, CVV,FK}. To prove these results we need to recall  some basic concepts concerning polyhedral cones. 
 
Recall that a closed cone $C$ in $\mathbb{R}^n$ is called a {\em polyhedral cone} if it is the intersection of finitely many closed half-spaces, i.e., there exist finitely many linear functionals $\phi_1,\ldots,\phi_k$ such that $C=\{x\in\mathbb{R}^n\colon \phi_i(x)\geq 0\mbox{ for all }i\}$. A subset $F$ of a polyhedral cone $C$ is called an {\em (exposed) face} if there exists a linear functional $\phi\in C^*$ such that $F =\{x\in C\colon \phi(x)=0\}$. A face $F$ of $C$ is called a {\em facet} if $\dim F =\dim C -1$.   Furthermore, it is a well-known fact from polyhedral geometry, that if $C$ is a polyhedral cone in $\mathbb{R}^n$ with $m$ facets and $C^\circ$ is nonempty, then there exist $m$ linear functional $\psi_1,\ldots,\psi_m$ such that 
\[
C = \{x\in\mathbb{R}^n\colon \psi_i(x)\geq 0\mbox{ for all }i\}.
\]
and each $\psi_i$ defines a facet of $C$. 

The {\em facet defining functionals}, $\psi_1,\ldots,\psi_m$, of a polyhedral cone $C$ correspond to  the extreme rays (1-dimensional faces) of the dual cone $C^*$ of $C$. If  we take $u\in C^\circ$, then we can normalize each facet defining functional $\psi_i$ so that 
$\psi_i(u)=1$. In that case we have  $\mathcal{E}_u=\{\psi_1,\ldots,\psi_m\}$. Now using Lemma \ref{lem:4.1} we can show the following  result regarding isometric embeddings  from \cite{N1}. 
\begin{theorem}\label{thm:polytope}
If $P$ is an $n$-dimensional polytope in $\mathbb{R}^n$ with $m$ facets, then the Hilbert geometry on $P^\circ$ can be isometrically embedded into $(\mathbb{R}^{m(m-1)/2}, \|\cdot\|_\infty)$. 
\end{theorem}
\begin{proof}
Let $C=\{(\lambda x,\lambda)\in\mathbb{R}^n\times\mathbb{R}\colon x\in P\mbox{ and }\lambda\geq 0\}$. Then $C$ is a polyhedral cone in $\mathbb{R}^{n+1}= \mathbb{R}^n\times \mathbb{R}$ with nonempty interior and $m$ facets. Denote the facet defining functionals of $C$ by $\psi_1,\ldots,\psi_m$. Furthermore, let $\Sigma^\circ = \{(x,s)\in C^\circ\colon s=1\}$. So, by Theorem \ref{thm:2.2} the Hilbert geometry $(P^\circ,\delta)$ is isometric to $(\Sigma^\circ,d)$.

Now define $\Psi\colon \Sigma^\circ\to \mathbb{R}^{m(m-1)/2}$ by 
\[
\Psi_{ij}(x) =\log\frac{\psi_i(x)}{\psi_j(x)}\mbox{\quad for all }1\leq i<j\leq m\mbox{ and }x\in \Sigma^\circ.
\]
It follows from Lemma \ref{lem:4.1} that 
$M(x/y) =\max_i \psi_i(x)/\psi_i(y)$ for all $x,y\in\Sigma^\circ$. 
So, 
\begin{eqnarray*}
d(x,y) & = & \log \Big{(}M(x/y)M(y/x)\Big{)}\\
     & = & \log \Big{(}\max_{i,j}\frac{\psi_i(x)\psi_j(y)}{\psi_i(y)\psi_j(x)}\Big{)}\\
     & = & \max_{i,j} \Big{(}\log\frac{\psi_i(x)}{\psi_j(x)} - \log \frac{\psi_i(y)}{\psi_j(y)}\Big{)}\\
       & = & \max_{1\leq i<j\leq m} \Big{|}\log\frac{\psi_i(x)}{\psi_j(x)} - \log \frac{\psi_i(y)}{\psi_j(y)}\Big{|},
\end{eqnarray*}
which shows that $d(x,y)=\|\Psi(x)-\Psi(y)\|_\infty$ for all $x,y\in \Sigma^\circ$. Thus, 
$\Psi$ is an isometry of $(\Sigma^\circ,d)$ into $(\mathbb{R}^{m(m-1)/2}, \|\cdot\|_\infty)$.
\end{proof}

The following result was essentially proved by Foertsch and Karlsson  \cite{FK}.
\begin{theorem}\label{thm:iso} 
A Hilbert geometry embeds isometrically into a finite-dimensional normed space if and only if its domain is the interior of a polytope. 
\end{theorem}
\begin{proof}
By the previous theorem it remains to show that a non-polytopal  Hilbert geometry cannot be isometrically embedded into any finite-dimensional normed space. 
For the sake of contradiction, assume that $(\Omega, \delta)$ is a non-polytopal Hilbert geometry and that $h$ is an isometry from $(\Omega,\delta)$ into a finite-dimensional normed space $(V,\|\cdot\|)$. 

The following result by Karlsson and Noskov \cite{KN} concerning the behavior of the Gromov product, \index{Gromov product},
\[
2(x\mid y)_p = \delta(x,p) +\delta(y,p) - \delta(x,y),
\]
 will be useful. 
Suppose that $p\in\Omega$ and $(x_k)$ and $(y_k)$ are sequences in $\Omega$ such that $x_k\to x\in\partial \Omega$, $y_k\to y\in\partial\Omega$, and the straight line segment $[x,y]\not\subset\partial \Omega$, then there exists a constant $R>0$ such that  
\[
\limsup_{k\to\infty}\, (x_k\mid y_k)_p\leq R.
\]
This result will be used to prove the following claim, from which the  contradiction will be derived. 

\noindent{\em Claim.} If there exist $\eta_1,\ldots,\eta_m\in\partial \Omega$ such that the straight line segment $[\eta_i,\eta_j]\not\subset\partial\Omega$ for all $i\neq j$, then there exist $v_1,\ldots,v_m\in V$ with $\|v_i\|=1$ for all $i$ and $\|v_i-v_j\|\geq 2$ for all $i\neq j$. 

To prove the claim let $p\in\Omega$ be fixed. Obviously the mapping $x\mapsto h(x) -h(p)$ is also an isometric embedding of $(\Omega,\delta)$ into $(V,\|\cdot\|)$. So, we may as well assume that $h(p)=0$. Now for $i=1,\ldots,m$ and $0\leq t< 1$ let 
$z_i(t)= (1-t)p +t\eta_i$. Note that the mapping $t\mapsto \delta(p,z_i(t))$ is continuous and $\lim_{t\to 1^-} \delta(p,z_i(t))=\infty$.   So, for each $i$ and $n\in\mathbb{N}$ there exists $0<t_{i,n}<1$ such that $\delta(p,z_i(t_{i,n})) =n$. For simplicity we write $z_i^n = z_i(t_{i,n})$. 

From the result by Karlsson and Noskov it follows that there exists a constant $M>0$ such that $\delta(z_i^n,z_j^n)\geq 2n -M$ for all $i\neq j $ and $n\geq N_0$, where $N_0$ is a sufficiently large integer. 
Define $u_i^n = \frac{1}{n}h(z_i^n)$. Then for each $i\neq j$ and $n\geq N_0$ we have 
\[
\|u_i^n-u_j^n\|= \frac{1}{n}\delta(z_i^n,z_j^n) \geq \frac{1}{n}(2n-M)= 2-\frac{M}{n}.
\] 
But also $\|u_i^n\| = \frac{1}{n}\|h(z_i^n)-h(p)\| = \frac{1}{n}\delta(p,z^n_i)=1$. 
As the unit sphere in $V$ is compact, we can find a subsequence $(u^{n_k}_i)$ converging to  some $v_i\in V$ for each $i$. Clearly the limits $v_1,\ldots,v_m$ satisfy $\|v_i\|=1$ for all $i$, and $\|v_i-v_j\|\geq 2$ for all $i\neq j$, which completes the proof of the claim.

To obtain a contradiction we will now show that there exist infinitely many points $\eta_1,\eta_2,\ldots\in\partial\Omega$ such that $[\eta_i,\eta_j]\not\subset\partial\Omega$ for all $i\neq j$, which, by the claim, violates the compactness of the unit sphere in $V$.  

Without loss of generality we may assume that $\Omega$ is contained in a finite-dimensional vector space $W$ and $0\in\Omega$. As the closure, 
$\overline{\Omega}$, of $\Omega$ is not a polytope, the polar of 
$\overline{\Omega}$, 
\[
\overline{\Omega}^* =\{\phi\in W^*\colon \phi(x)\leq 1\mbox{ for all }x\in \overline{\Omega}\},
\]
is also not a polytope. Thus, $\overline{\Omega}^*$ has infinitely many extreme points. As the exposed points are dense in the extreme points, see \cite{Stras}, there exist infinitely many exposed points $\xi_1,\xi_2,\ldots\in\partial \overline{\Omega}^*$.  Now using the fact that $\overline{\Omega}^{**} =\overline{\Omega}$  we can find $\eta_1,\eta_2,\ldots\in\partial \Omega$ such that  $\xi_i(\eta_i)=1$ for all $i$, and $\phi(\eta_i)<1$ for all $\phi \in \overline{\Omega}^*\setminus\{\xi_i\}$. Clearly, if $i\neq j$ and $0<t<1$, then 
$(1-t)\eta_i+ t\eta_j\in\Omega$, as $\phi((1-t)\eta_i+ t\eta_j)= (1-t)\phi(\eta_i)+t\phi(\eta_j)<1$ for all $\phi\in \overline{\Omega}^*$ and $0<t<1$. Thus, the segment $[\eta_i,\eta_j]\not\subset\partial\Omega$ for all $i\neq j$, which completes the proof.
\end{proof}
Theorem \ref{thm:iso} has been strengthened by Colbois and Verovic \cite{CV} to quasi-isometries, and by Bernig  \cite{Ber} and Colbois, Vernicos and Verovic \cite{CVV} to bi-lipschitz mappings.

 \subsection{Symmetric cones} 

Let $C^\circ$ be a symmetric cone in $(V,\langle \cdot\mid\cdot\rangle)$, and let $e\in C^\circ$ denote the unit in the associated Euclidean Jordan algebra on $V$. 
 An element $c\in V$ is called an {\em idempotent} \index{idempotents} if $c^2=c$. It is said to be a {\em primitive idempotent} \index{idempotents!primitive} if $c$ cannot be written as the sum of two non-zero idempotents. The set of all primitive idempotents in $V$ is denoted by $\mathcal{J}(V)$. 
 A set $\{c_1,\ldots,c_k\}$ is called a {\em complete system of orthogonal idempotents} \index{idempotents!complete system of orthogonal} if 
 \begin{enumerate}[(1)]
\item  $c_i^2= c_i$  for all $i$, 
\item $c_ic_j =0$ for all $i\neq j$, 
\item $c_1+\cdots+ c_k =e$. 
\end{enumerate}

The spectral theorem \cite[Theorem III.1.1]{FK} says that for each $x\in V$ there exist unique real numbers $\lambda_1, \ldots,\lambda_k$, all distinct, and a complete system of orthogonal idempotents $c_1,\ldots,c_k$ such that 
\[
x = \lambda_1 c_1+\cdots +\lambda_k c_k. 
\]
The numbers $\lambda_i$ are called the {\em eigenvalues} \index{eigenvalues} of $x$. The {\em spectrum} \index{spectrum} of $x$ is denoted  by 
\[
\sigma(x) =\{\lambda \colon \lambda \mbox{ eigenvalue of } x\}, 
\] 
and we write 
\[\lambda_+(x)=\max\{\lambda\colon\lambda\in\sigma(x)\} \mbox{\quad and\quad } 
\lambda_-(x)=\min\{\lambda\colon\lambda\in\sigma(x)\}.
\]

It can be shown, see \cite[Theorem III.2.1]{FK}, that $x\in C^\circ$ if and only if $\sigma(x)\subseteq (0,\infty)$. So, one can use the spectral decomposition,  $x =\lambda_1 c_1+\cdots +\lambda_kc_k$, of $x\in C^\circ$, to define the unique square root of $x$ by 
\[
x^{1/2} = \sqrt{\lambda_1}c_1+\cdots+\sqrt{\lambda_k}c_k.
\]
Similarly, functions $x\mapsto \log x$ and $x\mapsto x^t$ for $t\in\mathbb{R}$, can be defined on $C^\circ$. 
  
For $x\in V$   the linear mapping,
\[
P(x) = 2L(x)^2 -L(x^2),
\]
is called the {\em quadratic representation} \index{quadratic representation}  of $x$. (Recall that $L(x)\colon V\to V$ is the linear map given by $L(x)y =xy$.) Note that $P(x^{-1/2})x=e$ for all $x\in C^\circ$. It can also be shown that $P(x^{-1}) = P(x)^{-1}$ for all $x\in C^\circ$. In the example of  the Euclidean Jordan algebra on $\mathrm{Sym}_n$ the reader can verify that $P(A)B = ABA$. It is known, see \cite[Proposition III.2.2]{FK}, that if $x\in C^\circ$, then $P(x)\in\mathrm{Aut}(C)$, and hence 
\begin{equation}\label{eq:iso}
M(P(x) w/P(x)y) = M(w/y)\mbox{\quad and } m(P(x)w/P(x)y)=m(w/y)
\end{equation}
for all $w\in V$ and $x,y\in C^\circ$.
For $w\in V$ and $x\in C^\circ$ we write 
\[
\lambda_+(w,x) =\lambda_+(P(x^{-1/2}) w)\mbox{\quad and\quad }
\lambda_-(w,x) =\lambda_-(P(x^{-1/2}) w).
\]
The following formula for Hilbert's metric on symmetric cones was derived by Koufany \cite{Kou}.  
\begin{theorem}\label{thm:4.2}
If $V$ is a Euclidean Jordan algebra with symmetric cone $C^\circ$, then for 
$w\in V$ and $x\in C^\circ$,
\[
M(w/x) = \lambda_+(w,x)\mbox{\quad and\quad }
m(w/x) = \lambda_-(w,x).
\]
In particular, we have  for $x,y\in C^\circ$,
\[
d(x,y) = \log\Big{(} \frac{\lambda_+(x,y)}{\lambda_-(x,y)}\Big{)}.
\]
\end{theorem}
 \begin{proof}
 Let $z = P(x^{-1/2})w\in V$ and let $z=\lambda_1 c_1+\cdots +\lambda_kc_k$ be the spectral decomposition of $z$. Note that $w\leq_C \beta x$ is equivalent to $z=P(x^{-1/2})w\leq \beta P(x^{-1/2})x=\beta e$, since $P(x^{-1/2})\in\mathrm{Aut}(C)$. As $e=c_1+\cdots +c_k$, this inequality holds if and only if 
 $0\leq_C (\beta -\lambda_1)c_1+\cdots +(\beta -\lambda_k)c_k$, which is equivalent to  $\beta\geq \lambda_+(z)=\lambda_+(w,x)$. As $M(w/x)=\inf\{\beta\in\mathbb{R}\colon w\leq_C \beta x\}$, we deduce that $M(w/x) = \lambda_+(w,x)$. In the same way it can be shown that $m(w/x) =\lambda_-(w,x)$. 
 \end{proof}
In the example of the Euclidean Jordan algebra on $\mathrm{Sym}_n$ we find for 
$A,B\in\Pi_n(\mathbb{R})$ that 
\begin{eqnarray*}
\lambda_+(A,B) & = & \lambda_+(P(B^{-1/2})A)\\
 &  = & \max\{\lambda\colon \lambda\in\sigma(B^{-1/2}AB^{-1/2})\} \\
 & = & \max\{\lambda\colon \lambda\in\sigma(B^{-1}A)\}
\end{eqnarray*}
and 
\[
\lambda_-(A,B) =\min\{\lambda\colon \lambda\in\sigma(B^{-1}A)\}.
\]

\begin{remark}
If we combine Lemma \ref{lem:4.1} and Theorem \ref{thm:4.2} we find for $x\in V$ and $y$ in a symmetric cone $C^\circ$ that
\[
\lambda_+(x,y)= \max_{c\in \mathcal{E}_e}\frac{\langle x\mid c\rangle}{\langle y\mid c\rangle}
\mbox{\quad and \quad}
\lambda_-(x,y)= \min_{c\in \mathcal{E}_e}\frac{\langle x\mid c\rangle}{\langle y\mid c\rangle},
\] 
where $\mathcal{E}_e$ is the set of extreme points of $\Sigma_e=\{x\in C\colon \langle x\mid e\rangle =1\}$. It is known, see \cite[Proposition IV.3.2]{FK}, that 
\[
\mathcal{E}_e =\{x\in \Sigma_e\colon x\mbox{ is a primitive idempotent}\}.
\] 
So, 
\[
\lambda_+(x,y)= \max_{c\in \Sigma_e\cap\mathcal{J}(V)}\frac{\langle x\mid c\rangle}{\langle y\mid c\rangle}
\mbox{\quad and \quad}
\lambda_-(x,y)= \min_{c\in \Sigma_e\cap\mathcal{J}(V)}\frac{\langle x\mid c\rangle}{\langle y\mid c\rangle}.
\] 
These equalities are closely related to the min-max characterization of the eigenvalues of the elements of a Euclidean Jordan algebra by  Hirzebruch \cite{Hir}. 
\end{remark}

\section{Non-expansive mappings on Hilbert geometries} 
Many interesting examples of non-expansive mappings on Hilbert geometries arise as normalizations of  order-preserving, homogeneous mappings on cones.   For example, Bellman operators in Markov decision processes and Shapley operators in stochastic games are order-preserving and homogeneous mappings on $\mathbb{R}^n_+$  after a change of variables, see \cite{Bell,BeKo,Ney}.  These mappings, $f\colon \mathbb{R}^n_+\to\mathbb{R}^n_+$,  are of the form:
\[
f_i(x) = \inf_{\alpha\in A_i}\sup_{\beta\in B_i} r_i(\alpha,\beta)\Big{(}\prod_{j=1}^n x_j^{p_j(\alpha,\beta)}\Big{)},
\] 
for $1\leq i\leq n$ and $x\in\mathbb{R}^n_+$. Here $r_i(\alpha,\beta)\geq 0$, $\sum_j p_j(\alpha,\beta) =1$ and $0\leq p_j(\alpha,\beta)\leq 1$ for all $\alpha$, $\beta$, $i$, 
and $j$. The iterates of these operators are used to  compute the value of Markov decision processes and stochastic games. Other interesting examples of nonlinear order-preserving mappings on cones are so-called decimation-reproduction operators in the analysis of fractal diffusions \cite{LiN,Metz}, and $DAD$-operators in matrix scaling problems, see \cite{Men,N8}. 

In many applications it is important to understand the iterative behavior of such mappings, $f\colon C^\circ\to C^\circ$ and of the normalized mappings, $g\colon\Sigma_\phi^\circ\to\Sigma_\phi^\circ$, given by 
\[
g(x) = \frac{f(x)}{\phi(f(x))}\mbox{\quad for }x\in\Sigma_\phi^\circ=\{x\in C^\circ\colon \phi(x)=1\},
\] 
where $\phi\in (C^*)^\circ$.  The fact that these mappings are non-expansive with 
respect to Hilbert's metric is a very useful tool to analyze their dynamics. 

In the analysis of the dynamics of a non-expansive mapping $f$ on a Hilbert geometry $(\Omega,\delta)$ it is important to distinguish two cases: (1) $f$ has a fixed point in $\Omega$, and (2) $f$ does not have a fixed point in $\Omega$. In the first case, the limit points of each orbit of $f$ lie inside $\Omega$, whereas in the second case all the limit points of each orbit of $f$ lie inside $\partial\Omega$ by a result of Ca\l ka \cite{Cal}. In the next subsection we will consider the first case.

\subsection{Periodic orbits}
Before we get started we recall some basic notions from the theory of dynamical systems. 
A point $w\in\Omega$ is called a {\em periodic point} \index{periodic point} of $f\colon\Omega\to\Omega$ if the exists an integer $p\geq 1$ such that $f^p(w)=w$. The smallest such $p\geq 1$ is called the {\em period} \index{period} of $w$. In particular, $w\in\Omega$ is a {\em fixed point} \index{fixed point} if $f(w)=w$. The {\em orbit} of $x\in\Omega$ is given by $\mathcal{O}(x) =\{f^k(x)\colon k=0,1,2,\ldots\}$. We say that the orbit of $x\in\Omega$ {\em converges to a periodic orbit} if there exists a periodic point $w$ of $f$ with period $p$ such that $\lim_{k\to\infty} f^{kp}(x) =w$.  For $x\in\Omega$ we define the {\em $\omega$-limit set} \index{limit set} by 
\[
\omega(x;f)=\{y\in\overline{\Omega}\colon f^{k_i}(x)\to y\mbox{ for some subsequence } k_i\to\infty\}.
\]
Note that we allow the limit point $y$ to be in $\partial \Omega$ even though $f$ need not be defined there. 

The following result will play an important role. 
\begin{theorem}\label{thm:sup}
If $f\colon X\to X$ is a non-expansive mapping on a closed subset $X$ of $(\mathbb{R}^n,\|\cdot\|_\infty)$ and $f$ has a fixed point in $X$, then every orbit of $f$ converges to a periodic orbit whose period does not exceed $\max_k 2^k{n\choose k}$. 
\end{theorem}
A proof and a discussion of the history of this result can be found in \cite[Chapter 4]{LNBook}. 
The upper bound given in Theorem \ref{thm:sup} is currently the strongest known and was obtained by Lemmens and Scheutzow in \cite{LS2}. 
It was conjectured by Nussbaum in \cite{N1} that the optimal upper bound is $2^n$, but at present this has been confirmed for $n=1,2$ and $3$ only, see \cite{LyN}. 

Combining Theorem \ref{thm:sup} with the isometric embedding result in Theorem \ref{thm:polytope} we obtain the following corollary for non-expansive mappings on polytopal Hilbert geometries. 
\begin{corollary}\label{cor:5.2}
If $(\Omega,\delta)$ is a polytopal Hilbert geometry with $m$ facets and $f\colon D\to D$ is a non-expansive mapping on a closed subset $D$ of $(\Omega,\delta)$ with a fixed point, then each orbit of $f$ converges to a periodic orbit whose period does not exceed $\max_k 2^k {N\choose k}$, where $N=m(m-1)/2$.   
\end{corollary}
It is an interesting open problem to find the optimal upper bound for the possible periods of periodic points of non-expansive mappings on $(\Omega,\delta)$ in case $\Omega$ is the interior of an $n$-dimensional simplex. For the $2$-simplex, it was shown in \cite{Le} that $6$ is the optimal upper bound. It is believed that there exists a constant $c>2$ such that the periods do not exceed $c^n$ if $\Omega$ is an $n$-simplex, but this appears 
to be hard to prove. 

\begin{remark}  Corollary \ref{cor:5.2} has the following interesting geometric consequence: It is impossible to isometrically embed a Euclidean plane into any polytopal Hilbert geometry, as it is impossible to isometrically embed a rotation under irrational angle in such Hilbert geometries.
Thus, the Euclidean rank of a poytopal Hilbert geometry is $1$.  This observation  complements results by Bletz-Siebert and Foertsch \cite{BF}, who conjectured that the Euclidean rank of any Hilbert geometry is 1. 
\end{remark}

The following, more detailed, results exist for the possible periods of periodic points for order-preserving homogeneous mappings on polyhedral cones, see \cite{AGLN,LS2,LSp} and \cite[Chapter 8]{LNBook}. 
\begin{theorem}\label{thm:5.3}
If $f\colon C\to C$ is a continuous, order-preserving, homogeneous mapping on a polyhedral cone $C$ with nonempty interior and $m$ facets, then the following assertions hold: 
\begin{enumerate}[1.]
\item Every norm bounded orbit of $f$ converges to a periodic orbit whose period does not exceed 
\[
\frac{m!}{\lfloor\frac{m}{3}\rfloor! \lfloor\frac{m+1}{3}\rfloor! \lfloor\frac{m+2}{3}\rfloor!},
\]
where $\lfloor r\rfloor$ denoted the greatest integer $q\leq r$.
\item In case $C$ is an $n$-dimensional simplicial cone the set of possible periods of periodic points of $f$ is precisely the set of integers $p$ for which there exist integers $q_1$ and $q_2$ such that $p=q_1q_2$, $1\leq q_1\leq {k\choose\lfloor k/2\rfloor}$, and $1\leq q_2\leq {n\choose k}$ for some $0\leq k\leq n$.  
\item If $f(v)=\lambda v$ for some $\lambda>0$ and $v\in C^\circ$, then each orbit of the normalized mapping $g\colon \Sigma_\phi^\circ\to\Sigma_\phi^\circ$ given by 
\[
g(x) =\frac{f(x)}{\phi(f(x))}\mbox{\quad for }x\in\Sigma_\phi^\circ
\]
converges to a periodic orbit whose period does not exceed ${m\choose \lfloor m/2\rfloor}$. Moreover, the upper bound is sharp in case $C$ is a simplicial cone. 
\end{enumerate}
\end{theorem} 

In particular,  we see that on the cone $\mathbb{R}^3_+$, the set of possible periods of periodic points of order-preserving homogeneous mappings  $f\colon\mathbb{R}^3_+\to\mathbb{R}^3_+$ is $\{1,2,3,4,6\}$. So, it is impossible to have a period $5$ point in that case. An example of a mapping on $\mathbb{R}^3_+$ with a period $6$ orbit is the mapping 
$f\colon\mathbb{R}^3_+\to\mathbb{R}^3_+$ given by
\[
f\left (\begin{array}{c} x_1 \\ x_2 \\ x_3 \end{array}\right )= 
\left (\begin{array}{c} 
(3x_1\wedge x_2)\vee (3x_2\wedge x_3)\\
(3x_1\wedge x_3)\vee (3x_3\wedge x_2)\\
(3x_2\wedge x_1)\vee (3x_3\wedge x_1)\\
\end{array}\right )\mbox{\quad for $x\in\mathbb{R}^3_+$} 
\]
which has  $x=(1,2,0)$ as a period $6$ point. Here $a\wedge b =\min \{a,b\}$ and $a\vee b=\max\{a,b\}$ for $a,b\in\mathbb{R}$.

 In view of Corollary \ref{cor:5.2} it is interesting to ask the following question. 
 For which Hilbert geometries $(\Omega,\delta)$ do we have that the orbits of each non-expansive mapping $f\colon \Omega\to\Omega$  with a fixed point in $\Omega$ converge to periodic orbits?
 Obviously the answer is negative of $\Omega$ is the interior of an ellipsoid. However, the following was shown in \cite{Le}.
 \begin{theorem}\label{thm:5.4}
 If $(\Omega,\delta)$ is a strictly convex Hilbert geometry and there exists no $2$-dimensional affine plane $H$ such that $H\cap \Omega$ is the interior of an ellipsoid, then every orbit of a non-expansive mapping $f\colon\Omega\to\Omega$, with a fixed point in $\Omega$, converges to a periodic orbit. In fact, there exists an integer $q\geq 1$ such that $\lim_{k\to\infty} f^{kq}(x) $ exists for all $x\in\Omega$. 
 \end{theorem} 
 
\subsection{Denjoy-Wolff type theorems}
In this subsection we briefly discuss the behaviour of fixed point free non-expansive mappings on Hilbert geometries. A more detailed overview  of this topic is given by  A. Karlsson in this volume. 

Recall that if $f\colon \Omega\to\Omega$ is a fixed point free non-expansive mapping on a Hilbert geometry $(\Omega,\delta)$, then the {\em attractor}  \index{attractor} of $f$, 
\[
\mathcal{A}_f=\bigcup_{x\in\Omega} \omega(x;f),
\]
is contained in $\partial \Omega$ by Ca\l ka's result \cite{Cal}. In that case it is interesting to understand the structure of $\mathcal{A}_f$ in $\partial \Omega$. 
This problem was considered by Beardon in \cite{Bea1, Bea2}. He showed  that there is a striking resemblance between the dynamics of fixed point free mappings on Hilbert geometries and the dynamics of fixed point free analytic self-mappings of the open unit disc in the complex plane,  which is characterized by the classical Denjoy-Wolff theorem \cite{De,Wo1,Wo2}. 
\begin{theorem}[Denjoy-Wolff]  \index{Denjoy-Wolff theorem} If $\colon\mathbb{D}\to\mathbb{D}$ is a fixed point free analytic mapping on the open unit disc $\mathbb{D}$ in $\mathbb{C}$, then there exists a unique $\eta\in\partial \mathbb{D}$ such that $\lim_{k\to\infty} f^k(x) =\eta$ for all $x\in\mathbb{D}$, and the convergence is uniform on compact subsets of $\mathbb{D}$. 
\end{theorem}
Analytic self-mappings of the open unit disc are non-expansive under the Poincar\'e metric by the Schwarz-Pick lemma. In \cite{Bea2} Beardon noted that the Denjoy-Wolff theorem should be viewed as a result in geometry, as it essentially only depends on the hyperbolic properties of the Poincar\'e metric  on $\mathbb{D}$. In fact, he showed that the Denjoy-Wolff theorem can be generalized to fixed point free non-expansive mappings on metric spaces that possess sufficient ``hyperbolic'' properties. As a particular consequence of his results he obtained in \cite{Bea2} the following Denjoy-Wolff theorem result  for strictly convex Hilbert geometries. 
\begin{theorem} If $f\colon\Omega\to\Omega$ is a fixed point free non-expansive mapping on a strictly convex Hilbert geometry $(\Omega,\delta)$, then there exists a unique $\eta\in\partial \Omega$ such that $\lim_{k\to\infty} f^k(x) =\eta$ for all $x\in\Omega$, and the convergence is uniform on compact subsets of $\Omega$. 
\end{theorem}   

Beardon's arguments were sharpened by Karlsson in \cite{Ka}. It is known that Beardon's result does not hold for general Hilbert geometries. In fact,  for the Hilbert geometry on the open $2$-simplex, $\Delta_2$, Lins \cite{Li1} showed that if $S$ is a convex subset of $\partial \Delta_2$, then there exists a fixed point free non-expansive mapping $f\colon\Delta_2^\circ\to\Delta_2^\circ$ such that 
\[
S = \bigcup_{x\in\Omega} \omega(x;f).
\]
It was conjectured, however, by Karlsson and Nussbaum that the following is true, see \cite{N8}. 
\begin{conjecture}[Karlsson-Nusbaum] \label{con:KN}\index{Karlsson-Nussbaum conjecture} If $(\Omega,\delta)$ is a Hilbert geometry and $f\colon \Omega\to\Omega$ is a fixed point free non-expansive mapping, then there exists a convex set $\Lambda$ in $\partial \Omega$ such that $\omega(x,f)\subseteq \Lambda$ for all $x\in\Omega$
\end{conjecture} 
It turns out that to prove the conjecture it suffices to show that there exists $x\in\Omega$ such that the convex hull of $\omega(x;f)$ is contained in $\partial \Omega$, see \cite{N8}.    
At present there exist only partial results for Conjecture \ref{con:KN}. To begin there exists the following result by Lins \cite{Li1}. 
\begin{theorem}
If $(\Omega,\delta)$ is a polytopal Hilbert geometry and $f\colon \Omega\to\Omega$ is a fixed point free non-expansive mapping, then there exists $\Lambda\subset\partial\Omega$ such that $\omega(x;f)\subseteq \Lambda$ for all $x\in\Omega$. 
\end{theorem} 
 To prove this theorem Lins used the fact that a  polytopal Hilbert geometry can be isometrically embedded into a finite-dimensional normed space.  This property allowed him to show that for each $x\in\Omega$ there exists a horofunction $h\colon\Omega\to\mathbb{R}$ such that $\lim_{k\to\infty} h(f^k(x))=-\infty$.  
For general Hilbert geometries such a horofunction does not always exist, see \cite[Remark 3.2]{Linsthesis}. So, there seems to be no apparent  way to generalize Lins' arguments to the general case. 

 The following partial result for Conjecture \ref{con:KN} is due Karlsson \cite{Ka}. 
 \begin{theorem}\label{thm:5.10} If $f\colon \Omega\to\Omega$ is a fixed point free non-expansive mapping on a Hilbert geometry $(\Omega,\delta)$ and there exists $z\in \Omega$ such that 
 \[
 \lim_{k\to\infty} \frac{\delta(f^k(z),z)}{k} >0, 
 \]
then there exists $\Lambda\subseteq\partial \Omega$ convex such that $\omega(x;f)\subseteq\Lambda$ for all $x\in\Omega$. 
 \end{theorem}
 In \cite{KN} Karlsson and Noskov showed the following result, which says that the attractor $\mathcal{A}_f$ of a fixed point free non-expansive mapping $f$ on a Hilbert geometry  $(\Omega,\delta)$ is a star-shaped subset of $\partial \Omega$. 
 \begin{theorem}
  If $f\colon \Omega\to\Omega$ is a fixed point free non-expansive mapping on a Hilbert geometry $(\Omega,\delta)$, then there exists $\eta\in\partial\Omega$ such that for $y\in\mathcal{A}_f$ the straight line segment $[y,\eta]$ is contained in $\partial\Omega$. 
 \end{theorem} 
 The following counterpart to Theorem \ref{thm:5.10} was proved by Nussbaum \cite{N8}. 
\begin{theorem}  If $f\colon \Omega\to\Omega$ is a fixed point free non-expansive mapping on a Hilbert geometry $(\Omega,\delta)$ and there exists $z\in\Omega$ such that 
\[
\liminf_{k\to\infty} \delta (f^{k+1}(z),f^k(z)) =0,
\] 
then then there exists $\Lambda\subseteq\partial \Omega$ convex such that $\omega(x;f)\subseteq\Lambda$ for all $x\in\Omega$. 
\end{theorem}  
Other Denjoy-Wolff type theorems for  finite and infinite-dimensional 
Hilbert geometries can be found in \cite{GV,LNBook,LiN,N8}. Despite numerous efforts the Karlsson-Nussbaum conjecture remains one of the most outstanding problems in the field.


\frenchspacing
\begin{thebibliography}{1}

\bibitem{AGLN} M. Akian, S. Gaubert, B. Lemmens, and R.D. Nussbaum,
Iteration of order preserving subhomogeneous maps on a cone. 
\emph{Math. Proc. Cambridge Philos. Soc.}, 140(1):157--176, 2006.


\bibitem{Bau} F.L. Bauer. An elementary proof of the Hopf inequality for positive operators. \emph{Numer. Math.}, 7:331--337, 1965.

\bibitem{Bea1} A.F. Beardon, Iteration of contractions and analytic maps. 
\emph{J. Lond. Math. Soc. (2)}, 41(1):141--150, 1990.

\bibitem{Bea2}  A.F. Beardon, The dynamics of contractions. 
\emph{Ergodic Theory \& Dynam. Systems}, 17(6):1257--1266, 1997.


\bibitem{Bell} R. Bellman, 
{\em Dynamic Programming}. Princeton Univ. Press, Princeton, NJ, 1957.

\bibitem{Ber} A. Bernig,  Hilbert geometry of polytopes. {\em Arch. Math. (Basel)},  
92(4):314--324, 2009. 

\bibitem{BeKo} T. Bewley and E. Kohlberg, The asymptotic theory of stochastic games. {\em Math. Oper. Res.}, 1(3):197--208, 1976. 


\bibitem{Bi} G. Birkhoff, Extensions of Jentzsch's theorem.
\emph{Trans. Amer. Math. Soc.}, {85}(1):219--227, 1957.

\bibitem{BF} O. Bletz-Siebert and T. Foertsch, 
The Euclidean rank of Hilbert geometries. 
{\em Pacific J. Math.}, 231(2):257--278, 2007.
 
\bibitem{Bon} F.F. Bonsall, Linear operators in complete positive cones. 
\emph{Proc. Lond. Math. Soc.},  8(3):53--75, 1958. 

\bibitem{Bu1} P.J. Bushell, Hilbert's projective metric and positive contraction mappings in a Banach space. \emph{Arch. Ration. Mech. Anal.},  52: 330--338, 1973. 

\bibitem{Bu2} P.J. Bushell, On the projective contraction ratio for positive linear mappings. \emph{J. Lond. Math. Soc.}, 6:256--258, 1973. 

\bibitem{Cal} A. Ca\l ka, On a condition under which isometries have bounded orbits. \emph{Colloq. Math.}, 48:219--227, 1984.

\bibitem{CV} B. Colbois and P. Verovic,  Hilbert domains that admit a quasi-isometric embedding into Euclidean space. {\em Adv. Geom.}, 11(3):465--470, 2011.

\bibitem{CVV} B. Colbois, C. Vernicos, and P.  Verovic, 
Hilbert geometry for convex polygonal domains. {\em J. Geom.}, 100(1--2):37--64, 2011. 

\bibitem{dlH}  P. de la Harpe,  On Hilbert's metric for simplices. In {\em  Geometric group theory, Vol. 1 (Sussex, 1991)}, pp. 97Ð119, London Math. Soc. Lecture Note Ser., {\bf 181}, Cambridge Univ. Press, Cambridge, 1993. 



\bibitem{De} A. Denjoy,
Sur l'it\'eration des fonctions analytiques. 
\emph{C. R. Acad. Sc. Paris, S\'erie 1,},  182:255--257, 1926. 

\bibitem{EN1} S.P. Eveson and R.D. Nussbaum, 
An elementary proof of the Birkhoff-Hopf theorem. 
\emph{Math. Proc. Cambridge Philos. Soc.}, 117:31--55, 1995. 

\bibitem{EN2} S.P. Eveson and R.D. Nussbaum, 
Applications of the Birkhoff-Hopf theorem to the spectral theory of positive linear operators. \emph{Math. Proc. Cambridge Philos. Soc.}, 117:491--512, 1995. 

\bibitem{FK} J. Faraut and A. Kor\'anyi, {\em Analysis on Symmetric Cones}. 
Oxford Mathematical Monographs, Clarendon Press, Oxford, 1994.

\bibitem{FoK} T. Foertsch and A. Karlsson, Hilbert metrics and Minkowski norms. 
\emph{J. Geom.}, 83(1-2): 22--31, 2005.

\bibitem{GG} S. Gaubert and J. Gunawardena, The Perron-Frobenius
theory for homogeneous, monotone functions. \emph{Trans. Amer.
Math. Soc.},  356(12): 4931--4950, 2004.

\bibitem{GV} S. Gaubert and G. Vigeral, A maximin characterisation of the escape rate of non-expansive mappings in metrically convex spaces. {\em Math. Proc. Cambridge Philos. Soc.}, 152(2):341--363, 2012. 

\bibitem{Hilbert} D. Hilbert, 
\"{U}ber die gerade Linie als k\"urzeste Verbindung zweier Punkte. 
\emph{Math. Ann.}, 46:91--96, 1895. 

\bibitem{Hir} U. Hirzebruch, Der Min-Max-Satz von E. Fischer f\"ur formal real Jordan-Algebren, {\em Math. Ann.}, 186:65--69, 1970. 

\bibitem{Ho1} E. Hopf, An inequality for positive integral operators. 
\emph{J. Math. Mech}, 12:683--692, 1963.

\bibitem{Ho2} E. Hopf, Remarks on my paper "An inequality for positive integral operators". \emph{J. Math. Mech.},12:889--892, 1963.

\bibitem{Ka} A. Karlsson, Nonexpanding maps and Busemann functions.
\emph{Ergodic Theory \& Dynam. Systems}, 21(5):1447--1457, 2001.

\bibitem{KN} A. Karlsson and G.A. Noskov, The Hilbert metric and Gromov hyperbolicity. \emph{Enseign. Math.\,(2)}, 48(1-2):73--89, 2002.

\bibitem{Koe} M. Koecher, Positivit\"atsbereiche im $\mathbb{R}^n$. 
{\em Amer. J. Math.}, 79:575Ð-596, 1957. 

\bibitem{KP} E. Kohlberg and J.W. Pratt, 
The contraction mapping approach to the Perron-Frobenius theory: Why Hilbert's metric?, {\em Math. Oper. Res.}, 7(2):198--210, 1982.

\bibitem{Kou}  K. Koufany, Application of Hilbert's projective metric on symmetric cones. {\em Acta Math. Sin. (Engl. Ser.)}, 22(5):1467--1472,  2006.

\bibitem{Kras1} M.A. Krasnosel'skii, 
{\em Positive Solutions of Operator Equations.} 
P.Noordhoff Ltd., Groningen, The Netherlands, 1964.

\bibitem{KLS} M.A. Krasnosel'skii, J. A. Lifshits, and A.V. Sobolev. 
\emph{Postive Linear Systems: the Method of Positive Operators.}
Sigma  Ser.  Appl. Math. 5, Heldermann Verlag, 1989. 



\bibitem{Krau} U. Krause,  A nonlinear extension of the Birkhoff-Jentzsch theorem. 
{\em J. Math. Anal. Appl.}, 114(2):552--568, 1986. 

\bibitem{KR}  M.G. Kre\u\i n and M.A. Rutman, 
Linear operators leaving invariant a cone in a Banach space.  
\emph{Amer. Math. Soc. Translation  1950}, 26, 1950.

\bibitem{Le} B. Lemmens, Nonexpansive mappings on Hilbert's metric spaces. 
 \emph{Topol. Methods Nonlinear Anal.}, 38(1):45--58,  2011. 

 \bibitem{LS2} B. Lemmens and M. Scheutzow, On the dynamics of sup-norm
nonexpansive maps. \emph{Ergodic Theory \& Dynam. Systems}, 25(3):861--871, 2005. 

\bibitem{LSp} B. Lemmens and C. Sparrow, 
A note on periodic points of order preserving subhomogeneous maps. 
\emph{Proc. Amer. Math. Soc.}, 134(5):1513--1517, 2006.

 
 \bibitem{LNBook} B. Lemmens and R. Nussbaum, {\em Nonlinear Perron-Frobebius theory}. Cambridge Tracts in Mathematics 189, Cambridge Univ. Press, Cambridge, 2012. 

\bibitem{Lim} Y. Lim, Hilbert's projective metric on Lorentz cones and Birkhoff formula for Lorentzian compressions. {\em Linear Algebra Appl.}, 423(2--3):246--254, 2007.

\bibitem{Linsthesis} B. Lins, {\em  Asymptotic Behavior and Denjoy-Wolff Theorems for Hilbert Metric Nonexpansive Maps}. PhD. thesis, Rutgers University, New Brunswick, New Jersey, USA, 2007. 

\bibitem{Li1} B. Lins, A Denjoy-Wolff theorem for Hilbert metric nonexpansive maps on a polyhedral domain. \emph{Math. Proc. Cambridge. Philos. Soc.}, 143(1):157--164,  2007.  

\bibitem{LiN} B. Lins and R.D. Nussbaum, 
Denjoy-Wolff theorems, Hilbert metric nonexpansive maps and reproduction-decimation operators. {\em J. Funct. Anal.}, 254(9):2365--2386, 2008.


  
\bibitem{Lo}  K. L\"owner,
\"{U}ber monotone Matrixfunktionen. {\em Math. Z.},  38:177--216, 1934. 

\bibitem{LyN} R. Lyons and R.D. Nussbaum,
On transitive and commutative finite groups of isometries, in
\emph{Fixed Point Theory and Applications} (K-K. Tan ed.), World
Scientific: Singapore,  pp. 189--228, (1992).

\bibitem{Men} M.V. Menon, 
Some spectral properties of an operator associated with a pair of nonnegative matrices. {\em Trans. Amer. Math. Soc.}, 132:369--375, 1968. 

\bibitem{Metz} V. Metz, Hilbert's projective metric on cones of Dirichlet forms, {\em J. Funct. Anal.}, 127:438Ð455, 1995.

\bibitem{Ney} A. Neyman, Stochastic games and nonexpansive maps. In A. Neyman and S. Sorin, editors, {\em Stochastic Games and Applications (Stony Brook, NY, 1999)}, pp. 397--415, NATO Sci. Ser. C Math. Phys. Sci. 570, Kluwer Acad. Publ., Dordrecht, 2003. 

\bibitem{Nmem1} R.D. Nussbaum, 
Hilbert's projective metric and iterated nonlinear maps.
{\em Mem. Amer. Math. Soc.}, 391:1--137, 1988.

\bibitem{N1} R.D. Nussbaum,
Omega limit sets of nonexpansive maps: finiteness and cardinality
estimates. \emph{Differential Integral Equations},  3(3):523--540, 1990.

\bibitem{Nu93} R.D. Nussbaum, Entropy minimization, Hilbert's projective metric, and scaling integral kernels. {\em  J. Funct. Anal.}, 115(1):45-99, 1993. 

\bibitem{Nu3} R.D. Nusbaum, Periodic points of positive linear operators and  Perron-Frobenius operators, \emph{Integral Equations Operator Theory}, 39:41--97,  2001.

\bibitem{N8} R.D. Nussbaum, fixed point theorems and Denjoy-Wolff theorems for Hilbert's projective metric in infinite dimensions. \emph{Topol. Methods  Nonlinear Anal.},    29(2):199--250, 2007.

\bibitem{NPVL}  R.D. Nussbaum, A. Priyadarshi, and S. Verduyn Lunel, Positive operators and Hausdorff dimension of invariant sets. {\em Trans. Amer. Math. Soc.},  364(2):1029--1066, 2012.


\bibitem{PT1} A. Papadopoulos and M. Troyanov,
Weak Finsler structures and the Funk weak metric. 
{\em Math. Proc. Cambridge Philos. Soc.}, 147(2):419--437, 2009 . 

\bibitem{Per1} O. Perron, Grundlagen f\"ur eine Theorie des Jacobischen 
Kettenbruchalgorithmus. \emph{Math. Ann.}, 64:1--76,  1907. 

\bibitem{Per2} O. Perron, Zur Theorie der \"Uber Matrizen. \emph{Math. Ann.}, 64:248--263, 1907.
  
  \bibitem{Pha} B.B. Phadke, 
A triangular world with hexagonal circles. 
{\em Geometriae Dedicata}, 3:511--520, 1974/75. 

\bibitem{Po} A.J.B. Potter, 
Applications of Hilbert's projective metric to certain classes of non- homogeneous operators. {\em Quart. J. Math. Oxford Ser. (2)}, 28:93--99,  1977.

\bibitem{Rock} R.T. Rockafellar, \emph{Convex Analysis}, Princeton Landmarks in Mathematics, Princeton, N.J., 1997.

\bibitem{Sa} H. Samelson, On the Perron-Frobenius theorem.
\emph{Michigan Math. J.}, 4:57--59, 1957.

\bibitem{Shap} L.S. Shapley, Stochastic games. {\em Proc. Natl. Acad. Sci. USA},  39:1095--1100, 1953.

\bibitem{Simon} B. Simon, {\em Convexity. An analytic viewpoint}. Cambridge Tracts in Mathematics, 187, Cambridge Univ. Press, Cambridge, 2011.


\bibitem{Stras} S. Straszewicz, 
\"{U}ber exponierte Punkte abgeschlossener Punktmengen. {\em Fund. Math.}, 24:139--143, 1935. 

\bibitem{Tho} A.C. Thompson, On certain contraction mappings in a partially
ordered vector space. \emph{Proc. Amer. Math. Soc.}, 14:438--443, 1963.

\bibitem{Vin} E.B. Vinberg, Homogeneous cones. {\em Soviet Math. Dokl.}, 1:787--790, 1960.
 
\bibitem{Wa2} C. Walsh, The horofunction boundary of the Hilbert geometry. {\em Adv. Geom.}, 8(4):503--529, 2008. 

\bibitem{Wo1} J. Wolff, 
Sur l'it\'eration des fonctions born\'ees.
\emph{C. R. Acad. Sc. Paris, S\'erie 1}, 182:200--201, 1926. 

\bibitem{Wo2} J. Wolff, 
Sur une g\'en\'eralisation d'un th\'eoreme de Schwartz. 
\emph{C. R. Acad. Sc. Paris, S\'erie 1}, 183:500--502, 1926. 

\bibitem{ZKP} P.P. Zabreiko, M.A. Krasnosel'skii, and Yu.V. Pokornyi, 
On a class of positive linear operators. \emph{Funct. Anal. Appl.}, 5:57--70, 1972. 

\end{thebibliography}
\end{document}